\documentclass[11pt]{amsart}
\usepackage{stmaryrd}
\usepackage{amsmath}
\usepackage{amsfonts}
\usepackage{amsthm}
\usepackage{txfonts}
\usepackage{hyperref}
\usepackage{graphicx,amssymb,amsfonts,amsmath,amscd}
\usepackage{    textcomp}
\usepackage[all,ps,cmtip]{xy}
\usepackage{amscd}
\usepackage{xspace}
\usepackage{mathrsfs}
\usepackage{ dsfont }

\newtheorem{thm}{Theorem}[section]
\newtheorem{prop}[thm]{Proposition}
\newtheorem{lemma}[thm]{Lemma}
\newtheorem{cor}[thm]{Corollary}
\newtheorem{conj}[thm]{Conjecture}

\newtheorem{def-prop}[thm]{Definition-Proposition}
\newtheorem{prop-def}[thm]{Proposition-Definition}

\theoremstyle{definition} 
\newtheorem{defi}[thm]{Definition}
\newtheorem{rmk}[thm]{Remark}

\newcommand{\C}{{\bf C}}
\newcommand{\R}{{\bf R}}
\newcommand{\Z}{{\bf Z}}
\newcommand{\N}{{\bf N}}
\newcommand{\Q}{{\bf Q}}

\newcommand{\sF}{{\mathscr F}}

\newcommand{\sL}{{\mathscr L}}
\newcommand{\sM}{{\mathscr M}}
\newcommand{\sO}{{\mathscr O}}
\newcommand{\sP}{{\mathscr P}}

\newcommand{\sS}{{\mathscr S}}

\newcommand{\sV}{{\mathscr V}}

\newcommand{\sX}{{\mathscr X}}
\newcommand{\sY}{{\mathscr Y}}
\newcommand{\sZ}{{\mathscr Z}}

\newcommand{\ie}{\textit {i.e.}~}
\newcommand{\cf}{\textit {cf.}~}

\newcommand{\etc}{\textit {etc.}~}

\newcommand{\alb}{\mathop{\rm alb}\nolimits} 
\newcommand{\Alb}{\mathop{\rm Alb}\nolimits} 
\newcommand{\Aut}{\mathop{\rm Aut}\nolimits} 
\newcommand{\Bl}{\mathop{\rm Bl}\nolimits} 
\newcommand{\CH}{\mathop{\rm CH}\nolimits} 
\newcommand{\codim}{\mathop{\rm codim}\nolimits} 
\newcommand{\dd}{\mathrm{d}} 
\newcommand{\dual}{\mathop{^\vee}\nolimits} 
\newcommand{\End}{\mathop{\rm End}\nolimits} 
\newcommand{\Gr}{\mathop{\rm Gr}\nolimits} 
\newcommand{\Hom}{\mathop{\rm Hom}\nolimits}
\newcommand{\id}{\mathop{\rm id}\nolimits} 
\newcommand{\im}{\mathop{\rm Im}\nolimits} 
\newcommand{\Ker}{\mathop{\rm Ker}\nolimits} 
\renewcommand{\P}{\mathop{\bf P}\nolimits} 
\newcommand{\Pic}{\mathop{\rm Pic}\nolimits} 
\newcommand{\pr}{\mathop{\rm pr}\nolimits} 

\newcommand{\pt}{\mathop{\rm pt}\nolimits} 
\newcommand{\rank}{\mathop{\rm rank}\nolimits}
\newcommand{\Res}{\mathop{\rm Res}\nolimits} 
\newcommand{\Spec}{\mathop{\rm Spec}\nolimits}
\newcommand{\Sym}{\mathop{\rm Sym}\nolimits} 
\newcommand{\vide}{{\varnothing}} 

\renewcommand{\bar}{\overline}

\newcommand{\surj}{\twoheadrightarrow}

\newcommand{\lra}{\xrightarrow}
\newcommand{\dra}{\dashrightarrow}
\newcommand{\cart}{\ar@{}[dr]|\square} 
\renewcommand{\dual}{^{\vee}} 
\newcommand{\isom}{\simeq} 
\newcommand{\indlim}{\mathop{\varinjlim}}
\renewcommand{\tilde}{\widetilde}

\newcommand{\h}{\mathop{\mathfrak {h}}\nolimits}
\newcommand{\CHM }{\mathop{\rm {CHM}}\nolimits}
\newcommand{\SmProj }{\mathop{\rm {SmProj}}\nolimits}
\renewcommand{\1}{\mathop{\mathds{1}}\nolimits} 

\begin{document}

\title{On the action of symplectic automorphisms on
the $\CH_0$-groups of some hyper-K\"ahler fourfolds}
\author{Lie Fu}
\address{Institut Camille Jordan, Universit\'e Claude Bernard Lyon 1, 43 boulevard du 11 novembre 1918,
69622 Villeurbanne cedex, France }
\email{fu@math.univ-lyon1.fr}
\begin{abstract}
We prove that for any polarized symplectic automorphism of the Fano
variety of lines of a smooth cubic fourfold (equipped with the
Pl\"ucker polarization), the induced action on the Chow group of
0-cycles is identity, as predicted by Bloch-Beilinson conjecture. We
also prove the same result for the Chow group of homologically
trivial 2-cycles up to torsion.
\end{abstract}
\maketitle

\setcounter{section}{-1}
\section{Introduction}\label{sect:intro}
In this paper we are interested in an analogue of Bloch's conjecture
for the action on 0-cycles of a symplectic automorphism of a
irreducible holomorphic symplectic variety. First of all, let us
recall the Bloch conjecture and the general philosophy of the
Bloch-Beilinson conjecture which motivate our result.

The Bloch conjecture for 0-cycles on algebraic surfaces states the
following (\cf \cite[Page 17]{MR2723320}):
\begin{conj}[Bloch]\label{conj:surface}
Let $Y$ be a smooth projective variety, $X$ be a smooth projective
surface and $\Gamma\in \CH^2(Y\times X)$ be a correspondence between
them. If the cohomological correspondence $[\Gamma]^*:H^{2,0}(X)\to
H^{2,0}(Y)$ vanishes, then the Chow-theoretic correspondence
$$\Gamma_*: \CH_0(Y)_{\alb}\to \CH_0(X)_{\alb}$$ vanishes as well,
where $\CH_0(\bullet)_{\alb}:=\Ker\left(\alb:
\CH_0(\bullet)_{hom}\to \Alb(\bullet)\right)$ denotes the group of
the 0-cycles of degree 0 whose albanese classes are trivial.
\end{conj}
The special case in Bloch's conjecture where $X=Y$ is a surface $S$
and $\Gamma=\Delta_S\in\CH^2(S\times S)$  states: \emph{if a smooth
projective surface $S$ admits no non-zero holomorphic 2-forms, \ie
$H^{2,0}(S)=0$, then $\CH_0(S)_{\alb}=0$.} This has been proved for
surfaces which are not of general type \cite{MR0435073}, for
surfaces rationally dominated by a product of curves (by Kimura's
work \cite{MR2107443} on the nilpotence conjecture, \cf\cite[Theorem
3.30]{MR3186044}), and for Catanese surfaces and Barlow surfaces
\cite{MR3229054}, \etc

What is more related to the present paper is another particular
case of Bloch's conjecture: \emph{let $S$ be a smooth projective
surface with irregularity $q=0$. If an automorphism 
$f$ of $S$ acts on $H^{2,0}(S)$ as identity, \ie it preserves any
holomorphic 2-forms, then $f$ also acts as identity on $CH_0(S)$.}
This version is obtained from Conjecture \ref{conj:surface} by
taking $X=Y=S$ a surface and $\Gamma=\Delta_S-\Gamma_f\in
\CH^2(S\times S)$, where $\Gamma_f$ is the graph of $f$. We would
like to remark that it is also a consequence of the more general
Bloch-Beilinson-Murre conjecture.

Recently Voisin \cite{MR3007678} and
Huybrechts \cite{MR3008427} proved this conjecture for any
symplectic automorphism of finite order of a projective K3 surface (see also
\cite{MR3084558}):

\begin{thm}[Voisin, Huybrechts]\label{thm:VoiHuy}
 Let $f$ be an automorphism of finite order of a projective K3 surface $S$. If $f$ is symplectic, \ie $f^*(\omega)=\omega$, where $\omega$ is a generator of $H^{2,0}(S)$, then $f$ acts as identity on $\CH_0(S)$.
\end{thm}

The purpose of the paper is to investigate the analogous results in
higher dimensional situation. The natural generalizations of K3
surfaces in higher dimensions are the so-called \emph{irreducible
holomorphic symplectic varieties} or \emph{hyperk\"ahler manifolds}
(\cf \cite{MR730926}), which by definition is a simply connected
compact K\"ahler manifold with $H^{2,0}$ generated by a symplectic
form (\ie nowhere degenerate holomorphic 2-form). We can conjecture
the following vast generalization of Theorem \ref{thm:VoiHuy}:

\begin{conj}\label{conj:main}
Let $f$ be an automorphism of finite order of an irreducible
holomorphic symplectic projective variety $X$. If $f$ is symplectic:
$f^*(\omega)=\omega$, where $\omega$ is a generator $H^{2,0}(X)$.
Then $f$ acts as identity on $\CH_0(X)$.
\end{conj}

Like Theorem \ref{thm:VoiHuy} is predicted by Bloch's Conjecture
\ref{conj:surface}, Conjecture \ref{conj:main} is predicted by the
more general Bloch-Beilinson conjecture (\cf \cite{MR923131},
\cite[Chapitre 11]{MR2115000}, \cite{MR1265533}, \cite[Chapter
11]{MR1997577}). Instead of the most ambitious version involving the
conjectural category of mixed motives, let us formulate it only for
the 0-cycles and in the down-to-earth fashion (\cite[Conjecture
11.22]{MR1997577}), parallel to Conjecture \ref{conj:surface}:
\begin{conj}[Bloch-Beilinson]\label{conj:BB}
There exists a decreasing filtration $F^{\bullet}$ on
$\CH_0(X)_\Q:=\CH_0(X)\otimes \Q$ for each smooth projective variety
$X$, satisfying:
\begin{itemize}
 \item[$(i)$] $F^0\CH_0(X)_\Q=\CH_0(X)_\Q$,  $F^1\CH_0(X)_\Q=\CH_0(X)_{\Q, hom}$;
\item[$(ii)$] $F^\bullet$ is stable under algebraic correspondences;
\item[$(iii)$] Given a correspondence $\Gamma\in \CH^{\dim X}(Y\times X)_\Q$. If the cohomological correspondence $[\Gamma]^*: H^{i,0}(X)\to H^{i,0}(Y)$ vanishes, then the Chow-theoretic correspondence  $\Gr_F^i\Gamma_*: \Gr_F^i\CH_0(Y)_\Q\to \Gr_F^i\CH_0(X)_\Q$ on the $i$-th graded piece also
vanishes.
\item[$(iv)$] $F^{\dim X+1}\CH_0(X)_\Q=0$.
\end{itemize}
\end{conj}
The implication from the Bloch-Beilinson Conjecture \ref{conj:BB} to
Conjecture \ref{conj:main} is quite straightforward: as before, we
take $Y=X$ to be the symplectic variety. If $f$ is of order $n$,
then define two projectors in $\CH^{\dim X}(X\times X)_\Q$ by
$\pi^{inv}:=\frac{1}{n}\left(\Delta_X+\Gamma_f+\cdots+\Gamma_{f^{n-1}}\right)$
and $\pi^{\#}:=\Delta_X-\pi^{inv}$. Since $H^{2j-1, 0}(X)=0$ and
$H^{2j, 0}(X)=\C\cdot \omega^j$, the assumption that $f$ preserves
the symplectic form $\omega$ implies that $[\pi^\#]^*: H^{i,0}(X)\to
H^{i,0}(X)$ vanishes for any $i$. By $(iii)$, $\Gr_F^i(\pi^\#_*):
\Gr_F^i\CH_0(X)_\Q\to \Gr_F^i\CH_0(X)_\Q$ vanishes for any $i$. In
other words, for the Chow motive $(X,\pi^\#)$, $\Gr_F^i\CH_0(X,
\pi^\#)=0$ for each $i$. Therefore by five-lemma and the finiteness
condition $(iv)$, we have $\CH_0(X, \pi^\#)=0$, that is,
$\im(\pi^\#_*)=0$. Equivalently, for any $z\in\CH_0(X)_\Q$,
$\pi^{inv}(z)=z$, \ie $f$ acts as identity on $\CH_0(X)_\Q$. Thanks
to Roitman's theorem on the torsion of $\CH_0(X)$, the same still
holds true for $\Z$-coefficients.

In \cite{MR818549}, Beauville and Donagi provide an example of a
20-dimensional family of 4-dimensional irreducible holomorphic
symplectic projective varieties, namely the Fano varieties of lines
contained in smooth cubic fourfolds. In this paper, we propose to
study Conjecture \ref{conj:main} for finite order symplectic
automorphisms of this particular family. Our main result is the
following:

\begin{thm}\label{thm:main}
Let $f$ be an automorphism of a smooth cubic fourfold $X$. If the
induced action on its Fano variety of lines $F(X)$, denoted by $\hat
f$, preserves the symplectic form, then $\hat f$ acts on
$\CH_0(F(X))$ as identity.\\
Equivalently, the polarized symplectic automorphisms of
$F(X)$ act as identity on $\CH_0(F(X))$.
\end{thm}

%

We will show in \S\ref{sect:reduction} (\cf Corollary
\ref{cor:reduction}) how to deduce the above main theorem from the
following result:

\begin{thm}[\cf Theorem \ref{thm:main2}]\label{thm:realmain}
Let $f$ be an automorphism of a smooth cubic fourfold X acting as
identity on $H^{3,1}(X)$. Then $f$ acts as the identity on
$\CH_1(X)_\Q$.
\end{thm}

As a consequence of the main theorem, we will deduce in the last
section the following consequence:
\begin{cor}\label{cor:CH2}
Under the same hypothesis as in Theorem \ref{thm:main}: if $\hat f$
is a polarized symplectic automorphism of the Fano variety of lines
$F(X)$ of a smooth cubic fourfold $X$, then $\hat f$ acts on
$\CH_2(F(X))_{\Q,hom}$ as identity.
\end{cor}

Let us explain the main strategy of the proof of Theorem
\ref{thm:realmain}: we use the techniques of \emph{spread} as in Voisin's
paper \cite{MR3099982}. More precisely, we can summarize as follows
the main steps. Let $f$ and $X$ be as in Theorem \ref{thm:realmain}.
\begin{itemize}
  \item[(a)]Let $\Gamma_f\subset X\times X$ be the graph of $f$. Let $n$
be the order of $f$ and $\pi^{inv}:=\sum_{i=0}^{n-1}\Gamma_{f^i}\in
\CH^4(X\times X)_\Q$ be the projector onto the invariant part of
$X$. In order to prove that $f$ acts trivially on $\CH_1(X)_Q$ it
suffices to show that there exists a decomposition in $\CH^4(X\times
X)_\Q$:
\begin{equation}\label{eqn:toprove}
\Delta_X-\pi^{inv}=\Gamma'_0+Z'+Z'',
\end{equation}
where $\Gamma'_0$ is supported on $Y\times Y$ for a codimension 2
closed algebraic subset $Y\subset X$, and $Z', Z''$ are the
pull-back of cycles on $X\times \P^5$ and $\P^5\times X$
respectively, \cf (\ref{finaleqn2}).

\item[(b)] To prove (\ref{eqn:toprove}), we show firstly that there exists an
algebraic cycle $\Gamma'_0$ supported on  $Y\times Y$ for a codimension 2
closed algebraic subset $Y\subset X$, such that
$\Delta_X-\pi^{inv}=\Gamma'_0$ has zero cohomology class (with
rational coefficients), see Proposition \ref{prop:fiberwise}. Now
consider the family $\sX\to B$ of all smooth cubic fourfolds which
are mapped to themselves by the automorphism $f$ (here $f$ is a
fixed projective automorphism of $\P^5$). One shows that the cycle
$\Gamma'_0$ for the varying $X\times X$ fit together to give a cycle
$\Gamma'$ on $\sX\times_B\sX$, see Proposition \ref{prop:crucial}.
Of course, the cycles $(\Delta_{X_b}-\pi^{inv}_{X_b})$ fit together
to give a cycle $\Gamma$ on $\sX\times_b\sX$. Then the cohomology
class of $\Gamma-\Gamma'$ restricts to zero on each fiber $X_b\times
X_b$. By a Leray spectral sequence argument as in \cite{MR3099982},
there exist algebraic cycles $\sZ'$, $\sZ''$ on $\sX\times_b\sX$,
which are the pull-back of cycles on $\sX\times\P^5$ and $\P^5\times
\sX$ respectively, such that $\Gamma-\Gamma'-\sZ'-\sZ''$ has zero
cohomology class.

\item[(c)]Now comes the core of the proof: one show that given $z\in
\CH^4(\sX\times_B\sX)_\Q$ which is homologically trivial there
exists a dense open subset $B'\subset B$ such that the restriction
of $z$ to the base-changed family $\sX'\times_{B'}\sX'$ vanishes.
There are two main ingredients in the proof: \\
$(i)$ One completes the family $\sX\times_B\sX$ to a smooth
projective variety for which the rational equivalence and
cohomological equivalence coincide, once we tensor with $\Q$.\\
$(ii)$ To extend a homologically trivial cycle to a homologically
trivial cycle in the compactification (or rather its resolution of singularities), we exploit the fact that the
Chow motive of a cubic fourfold decomposes into pieces which do not
exceed the size of the Chow motives of surfaces.

\item[(d)] Applying the result in $(c)$ to the cycle
$(\Gamma-\Gamma'-\sZ'-\sZ'')$ gives (\ref{eqn:toprove}) and hence
the result.
\end{itemize}


We organize the paper as follows. In \S\ref{sect:setting}, we start
describing the parameter space of cubic fourfolds with an action
satisfying that the induced actions on the Fano varieties of lines
are symplectic. In \S\ref{sect:reduction}, the main theorem is
reduced to a statement about the 1-cycles of the cubic fourfold. By
varying the cubic fourfold, in \S\ref{sect:preparation} we reduce
the main theorem \ref{thm:main} to the form that we will prove,
which concerns only the 1-cycles of a general member in the family.
The purpose of \S\ref{sect:totalspace} is to establish the
triviality of Chow groups of some total spaces. The first half
\S\ref{sect:totalspace}.1 shows the triviality of Chow groups of its
compactification; then the second half \S\ref{sect:totalspace}.2
passes to the open part by comparing to surfaces. \S\ref{sect:proof}
proves the main Theorem \ref{thm:main} by combining the
strategy of Voisin's paper  \cite{MR3099982} and the
result of \S\ref{sect:totalspace}. In \S\ref{sect:natural} we
reformulate the hypothesis in the main theorem to the assumption of
being `polarized'. Finally in \S\ref{sect:CH2}, we verify another
prediction of Bloch-Beilinson conjecture on the Chow group of
2-cycles (Corollary \ref{cor:CH2}) from our main result.

We will work over the complex numbers throughout this paper.

 \noindent\textbf{Acknowledgements:} I would like to express my
 gratitude to my thesis advisor Claire Voisin for bringing to me this
 interesting subject as well as many helpful suggestions. I also want
 to thank Mingmin Shen for pointing out the connection of our result
 and his joint work with Charles Vial \cite{ShenVial}, which
 motivates the last section of the paper. Finally, I thank the referee for his or her very helpful suggestions which improved the paper a lot.

\section{Basic settings}\label{sect:setting}

In this first section, we establish the basic settings for
automorphisms of the Fano variety of a cubic fourfold, and work out
the condition corresponding to the symplectic assumption.

Let $V$ be a fixed 6-dimensional $\C$-vector space, and
$\P^5:=\P(V)$ be the corresponding projective space of 1-dimensional
subspaces of $V$. Let $X\subset\P^5$ be a smooth cubic fourfold,
which is defined by a polynomial $T\in H^0(\P^5,
\sO(3))=\Sym^3V\dual$. Let $f$ be an automorphism of $X$. Since
$\Pic(X)=\Z\cdot \sO_X(1)$, any automorphism of $X$ is
\emph{induced}: it is the restriction of a linear automorphism of
$\P^5$ preserving $X$, still denoted by $f$.

It is classical and well-known that $\Aut(X)$ is a finite group. Let $f$ be of order $n\in \N_+$.  Since the minimal polynomial of
$f$ is semi-simple, we can assume without loss of generality that
$f:\P^5\to \P^5$ is given by:
\begin{equation}\label{eqn:f}
  f:[x_0: x_1:\cdots: x_5] \mapsto [\zeta^{e_0}x_0:
  \zeta^{e_1}x_1:\cdots:
\zeta^{e_5}x_5],
\end{equation}
where $\zeta=e^{\frac{2\pi \sqrt{-1}}{n}}$ is a primitive $n$-th
root of unity and $e_i\in \Z/n\Z$ for $i=0,\cdots, 5$. Now it is
clear that $X$ is preserved by $f$ if and only if its defining
equation $T$ is contained in an eigenspace of $\Sym^3V\dual$, where
$\Sym^3V\dual$ is endowed with the induced action coming from $V$.

Let us make it more precise: as usual, we use the coordinates $x_0,
x_1, \cdots, x_5$ of $\P^5$ as a basis of $V\dual$, then
$\left\{\underline x^{~\underline \alpha}\right\}_{\underline
\alpha\in \Lambda}$ is a basis of $\Sym^3V\dual=H^0(\P^5, \sO(3))$,
where $\underline x^{~\underline \alpha}$ denotes
$x_0^{\alpha_0}x_1^{\alpha_1}\cdots x_5^{\alpha_5}$, and
\begin{equation}\label{eqn:Lambda}
\Lambda:=\left\{\underline \alpha=(\alpha_0, \cdots, \alpha_5)\in
\N^5~|~ \alpha_0+\cdots+\alpha_5=3\right\}.
\end{equation}
Therefore the eigenspace decomposition of $\Sym^3V\dual$ is the
following:
$$\Sym^3V\dual=\bigoplus_{j\in \Z/n\Z}\left(\bigoplus_{\underline\alpha\in \Lambda_j}\C\cdot\underline x^{~\underline
\alpha}\right),$$ where for each $j\in \Z/n\Z$, we define the subset
of $\Lambda$
\begin{equation}\label{eqn:Lambdaj}
\Lambda_j:=\left\{\underline \alpha=(\alpha_0, \cdots, \alpha_5)\in
\N^5~|~ \substack{\alpha_0+\cdots+\alpha_5=3\\
e_0\alpha_0+\cdots +e_5\alpha_5=j \mod n }\right\}.
\end{equation}
and the eigenvalue of $\bigoplus_{\underline\alpha\in
\Lambda_j}\C\cdot\underline x^{~\underline \alpha}$ is $\zeta^j$.
Therefore, explicitly speaking, we have:
\begin{lemma}\label{lemma:preserve}
Keeping the notation (\ref{eqn:f}), (\ref{eqn:Lambda}),
(\ref{eqn:Lambdaj}), the cubic fourfold $X$ is preserved by $f$ if
and only if there exists a $j\in \Z/n\Z$ such that the defining
polynomial $T\in \bigoplus_{\underline\alpha\in
\Lambda_j}\C\cdot\underline x^{~\underline \alpha}$.
\end{lemma}

Let us deal now with the symplectic condition for the induced action
on $F(X)$. First of all, let us recall some basic constructions and
facts. The following subvariety of the Grassmannian $\Gr(\P^1,
\P^5)$
\begin{equation}\label{eqn:F}
F(X):=\left\{[L]\in \Gr(\P^1, \P^5)~|~ L\subset X\right\}
\end{equation}
is called the \emph{Fano variety of lines}\footnote{In the
scheme-theoretic language, $F(X)$ is defined to be the zero locus of
$s_T\in H^0\left(\Gr(\P^1, \P^5), \Sym^3S\dual\right)$, where $S$ is
the universal tautological subbundle on the Grassmannian, and $s_T$
is the section induced by $T$ using the morphism of vector bundles
$\Sym^3V\dual\otimes \sO\to \Sym^3 S\dual$ on $\Gr(\P^1, \P^5)$.} of
$X$. It is well-known that $F(X)$ is a 4-dimensional smooth
projective variety equipped with the restriction of the Pl\"ucker
polarization of the ambient Grassmannian. Consider the incidence
variety (\ie the universal projective line over $F(X)$):
$$P(X):=\left\{(x, [L])\in X\times F(X)~|~ x\in L\right\}.$$
We have the following natural correspondence:
\begin{displaymath}
  \xymatrix{
  P(X)\ar[r]^{q} \ar[d]_{p} & X\\
  F(X) & \\
  }.
\end{displaymath}
\begin{thm}[Beauville-Donagi \cite{MR818549}]\label{thm:BD}
 Using the above notation,
\begin{itemize}
  \item[$(i)$] $F(X)$ is a 4-dimensional irreducible holomorphic
  symplectic projective variety, \ie $F(X)$ is simply-connected and
  $H^{2,0}(F(X))=\C\cdot\omega$ with $\omega$ a nowhere degenerate
  holomorphic 2-form.
  \item[$(ii)$] The correspondence $$p_*q^*: H^4(X,\Z)\to H^2(F(X), \Z)$$ is an
isomorphism of Hodge structures.
\end{itemize}
\end{thm}
In particular, $p_*q^*: H^{3,1}(X)\lra{\isom}H^{2,0}(X)$ is an
isomorphism. If $X$ is equipped with an action $f$ as before, we
denote by $\hat f$ the induced automorphism of $F(X)$. Since the
construction of the Fano variety of lines $F(X)$ and the
correspondence $p_*q^*$ are both functorial with respect to $X$, the
condition that $\hat f$ is \emph{symplectic}, namely $\hat
f^*(\omega)=\omega$ for $\omega$ a generator of $H^{2,0}(F(X))$), is
equivalent to the condition that $f^*$ acts as identity on
$H^{3,1}(X)$. Working this out explicitly, we arrive at the
following

\begin{lemma}\label{lemma:symp}
Let $f$ be the linear automorphism in (\ref{eqn:f}), and $X$ be a
cubic fourfold defined by equation $T$. Then the followings are
equivalent:
\begin{itemize}
 \item $f$ preserves $X$ and the
induced action $\hat f$ on $F(X)$ is symplectic;
 \item  There
exists a $j\in \Z/n\Z$ satisfying the equation
\begin{equation}\label{eqn:j}
e_0+e_1+\cdots+e_5=2j \mod n,
\end{equation}
such that the defining polynomial $T\in
\bigoplus_{\underline\alpha\in \Lambda_j}\C\cdot\underline
x^{~\underline \alpha}$, where as in (\ref{eqn:Lambdaj})
\begin{equation*}
  \Lambda_j:=\left\{\underline \alpha=(\alpha_0, \cdots, \alpha_5)\in
\N^5~|~ \substack{\alpha_0+\cdots+\alpha_5=3\\
e_0\alpha_0+\cdots +e_5\alpha_5=j \mod n }\right\}.
\end{equation*}
\end{itemize}
\end{lemma}
\begin{proof}
  Firstly, the condition that $f$ preserves $X$ is given in Lemma
  \ref{lemma:preserve}. As is remarked before the lemma, $\hat f$ is symplectic if and
  only if $f^*$ acts as identity on $H^{3,1}(X)$. On the other hand, by
  Griffiths' theory of the Hodge structures of hypersurfaces (\cf \cite[Chapter 18]{MR1997577}), $H^{3,1}(X)$ is
  generated by the residue $\Res{\frac{\Omega}{T^2}}$, where $\Omega:=\sum_{i=0}^5(-1)^ix_i\dd x_0\wedge\cdots\wedge\widehat{\dd
  x_i}\wedge\cdots\wedge \dd x_5$ is a generator of $H^0(\P^5,
  K_{\P^5}(6))$. $f$ given in (\ref{eqn:f}), we get
  $f^*\Omega=\zeta^{e_0+\cdots+e_5}\Omega$ and $f^*(T)=\zeta^jT$.
  Hence the action of $f^*$ on $H^{3,1}(X)$ is multiplication by
  $\zeta^{e_0+\cdots+e_5-2j}$, from which we obtain equation (\ref{eqn:j}).
\end{proof}

\section{Reduction to 1-cycles on cubic fourfolds}\label{sect:reduction}

The objective of this section is to prove Corollary
\ref{cor:reduction}. It allows us in particular to reduce the main
Theorem \ref{thm:main}, which is about the action on 0-cycles on the
Fano variety of lines, to the study of the action on 1-cycles of the
cubic fourfold itself (see Theorem \ref{thm:main2}).

To this end, we want to make use of Voisin's equality (see
Proposition \ref{prop:Voisin}$(ii)$) in the Chow group of 0-cycles
of the Fano variety of a cubic fourfold.  Let $X$ be a (smooth)
cubic fourfold, $F:=F(X)$ be its Fano variety of lines and $P:=P(X)$
be the universal projective line over $F$ fitting into the diagram
below:
\begin{displaymath}
 \xymatrix{
P \ar[d]_{p} \ar[r]^{q} & X\\
F & }
\end{displaymath}
For any line $L$ contained in $X$, we denote the corresponding point
in $F$ by $l$. Define $S_l:=\left\{l'\in F~|~L'\cap L\neq
\vide\right\}$ to be the surface contained in $F$ parameterizing
lines in $X$ meeting a give line $L$. As algebraic cycles,
\begin{equation}\label{eqn:lToL}
 L=q_*p^*(l) \in \CH_1(X);
\end{equation}
\begin{equation}\label{eqn:LToSl}
 S_l=p_*q^*(L) \in \CH_2(F).
\end{equation}
The following relations are discovered by Voisin in
\cite{MR2435839}:
\begin{prop}\label{prop:Voisin}
Let $I:=\left\{(l,l')\in F\times F~|~ L\cap L'\neq \vide\right\}$ be
the incidence subvariety. We denote by $g\in \CH^1(F)$ the Pl\"ucker
polarization, and by $c\in \CH^2(F)$ the second Chern class of the
restriction to $F$ of the tautological rank 2 subbundle on
$\Gr(\P^1, \P^5)$.
\begin{itemize}
 \item[$(i)$] There is a quadratic relation in $\CH^4(F\times F)$:
$$I^2=\alpha \Delta_F+ I\cdot \Gamma+ \Gamma',$$
where $\alpha\neq 0$ is an integer, $\Gamma$ is a degree 2
polynomial in $\pr_1^*g$, $\pr_2^*g$, and $\Gamma'$ is a weighted
degree 4 polynomial in $\pr_1^*g$, $\pr_2^*g$, $\pr_1^*c$,
$\pr_2^*c$.
 \item[$(ii)$] For any $l\in F$, we have an equality in $\CH_0(F)$:
\begin{equation}\label{eqn:SlTol}
 S_l^2=\alpha\cdot l+\beta S_l\cdot g^2+ \Gamma'',
\end{equation}
where $\alpha\neq 0$ and $\beta$ are constant integers, $\Gamma''$
is a polynomial in $g^2$ and $c$ of degree 2  with integral
coefficients independent of $l$.
\end{itemize}
\end{prop}
\begin{proof} For the first equality $(i)$, \cf \cite[Proposition 3.3]{MR2435839}. For $(ii)$, we restrict the relation in $(i)$ to a fiber $\{l\}\times F$, then $I|_{\{l\}\times F}=S_l$ and $\Delta_F|{\{l\}\times F}=l$. Hence the equation (\ref{eqn:SlTol}).
\end{proof}

\begin{cor}
Given an automorphism $f$ of a cubic fourfold $X$, let $L$ be a line
contained in $X$ and $l\in \CH_0(F)$, $S_l\in \CH_2(F)$ be the
cycles as above. Then the followings are equivalent:
\begin{itemize}
 \item[$(i)$] $\hat f(l)=l$ in $\CH_0(F)$ ;
 \item[$(ii)$] $f(L)=L$ in $\CH_1(X)$ ;
 \item[$(iii)$] $\hat f(S_l)=S_l$ in $\CH_2(F)$.
\end{itemize}
The same equivalences hold also for Chow groups with rational
coefficients.
\end{cor}
\begin{proof}
 $(i)\Rightarrow(ii)$: by (\ref{eqn:lToL}) and the functorialities of $p$ and $q$.\\
 $(ii)\Rightarrow(iii)$: by (\ref{eqn:LToSl}) and the functorialities of $q$ and $p$.\\
 $(iii)\Rightarrow(i)$: by (\ref{eqn:SlTol}) and the fact that $g$, $c$ are all invariant by $\hat f$, we obtain $\alpha\left(l-\hat f(l)\right)=0$ in $\CH_0(F)$ with $\alpha\neq 0 $. However by Roitman theorem $\CH_0(F)$ is torsion-free, thus $l=\hat f(l)$ in
 $\CH_0(F)$.\\
 Of course, the same\footnote{In fact easier, because we do not need
to invoke Roitman theorem.} proof gives the same equivalences for
Chow groups with rational coefficients.
\end{proof}

In particular, we have:
\begin{cor}\label{cor:reduction}
  Let $f$ be an automorphism of a cubic fourfold $X$ and $F$ be the Fano variety of lines of $X$, equipped with the induced action $\hat f$. Then the followings are
  equivalent:
  \begin{itemize}
 \item[$(i)$] $\hat f$ acts on $\CH_0(F)$ as identity;
 \item[$(ii)$] $\hat f$ acts on $\CH_0(F)_\Q$ as identity;
 \item[$(iii)$] $f$ acts on $\CH_1(X)$ as identity;
  \item[$(iv)$] $f$ acts on $\CH_1(X)_\Q$ as identity.
\end{itemize}
\end{cor}
\begin{proof}
  By the result in \cite{MR3187624}\footnote{For rational coefficients it can be easily deduced by the argument in \cite{MR1440062}.} that $\CH_1(X)$ is generated by the lines contained in
  $X$, the previous corollary gives the equivalences $(i)\Leftrightarrow
  (iii)$ and $(ii)\Leftrightarrow (iv)$. On the other hand, $F$
is simply-connected and thus its Albanese variety is trivial.
Therefore $\CH_0(F)$ is torsion-free by Roitman's theorem, hence
$(i)\Leftrightarrow (ii)$.
\end{proof}

We remark that this corollary allows us to reduced Theorem
\ref{thm:main} to Theorem \ref{thm:realmain} which is stated purely
in terms of the action on the Chow group of the 1-cycles of the
cubic fourfold itself.

\section{Reduction to the general member of the family}\label{sect:preparation}

Our basic approach to the main theorem \ref{thm:main} is to
\emph{vary} the cubic fourfold in family and make use of certain
good properties of the total space (\cf \S\ref{sect:totalspace}) to
get some useful information for a member of the family. To this end,
we give in this section a family version of previous constructions,
and then by combining with Corollary \ref{cor:reduction}, we reduce the
main theorem \ref{thm:main} to Theorem \ref{thm:main2} which is a
statement for 1-cycles of a \emph{general} member in the family.

Fix $n\in \N_+$, fix $f$ as in (\ref{eqn:f}) and fix a solution
$j\in \Z/n\Z$ of (\ref{eqn:j}). Consider the projective space
parameterizing certain possibly singular cubic hypersurfaces in
$\P^5$.
\begin{equation*}
\bar B= \P\left(\bigoplus_{\underline \alpha\in \Lambda_j}\C\cdot
\underline x^{~\underline \alpha}\right),
\end{equation*}
where $\Lambda_j$ is defined in (\ref{eqn:Lambdaj}).  Let us denote
the universal family by
\begin{displaymath}
  \xymatrix{
  \bar\sX\ar[d]^{\pi}\\
  \bar B
  }
\end{displaymath}
whose fibre over the a point $b\in \bar B$ is a cubic hypersurface
in $\P^5$ denoted by $X_b$. Let $B\subset \bar B$ be the Zariski
open subset parameterizing the smooth ones. By base change, we have
over $B$ the universal family of smooth cubic fourfolds with a
(constant) fiberwise action $f$, and similarly the universal Fano
variety of lines $\sF$ equipped with the corresponding fiberwise
action $\hat f$:
\begin{equation}\label{diag:universal}
\xymatrix{ \hat f\circlearrowright\sF\ar[dr] & \sX\ar[d]^{\pi}
\ar@{^(->}[r] &B\times
\P^5\circlearrowleft f\ar[dl]^{\pr_1}\\
& B & }
\end{equation}
The fibre over $b\in B$ of $\sF$ is denoted by $F_b=F(X_b)$, on
which $\hat f$ acts symplectically by construction.

By the following general fact, we claim that to prove the main
theorem \ref{thm:main}, it suffices to prove it for a very general
member in the family:
\begin{lemma}\label{lemma:ReductionToGeneral}
Let $\sF\to B$ be a smooth projective fibration with a fibrewise
action $\hat f$ (for example in the situation (\ref{diag:universal})
before). If for a general
point $b\in
B$, $\hat f$ acts as identity on $\CH_0(F_b)$, then the same thing
holds true for any $b\in B$.
\end{lemma}
\begin{proof}
  For any $b_0\in B$, we want to show that $\hat f$ acts as identity on
  $\CH_0(F_{b_0})$. Given any 0-cycle $Z\in \CH_0(F_{b_0})$, we can find a generically finite dominant base-change
  \begin{displaymath}
   \xymatrix{\sF'\ar[r]\cart \ar[d] & \sF\ar[d]\\
   B'\ar[r] & B}
  \end{displaymath} and a cycle $\sZ\in\CH_{\dim
  B'}(\sF')$, such that $\sZ|_{F'_{b'_0}}=Z$, where $b'_0\in B'$ is a preimage of $b_0\in B$. (For example, when $Z$ is just one point, we can take $B'$ to be the transversal intersection of ($\dim\sF-\dim B$) general hyperplane sections of $\sF$ passing through this point. By iterating this construction, one can treat the general case when $Z$ is a 0-cycle.)
   Hence $F'_{b_0'}=F_{b_0}$. Now consider $\Gamma:=\hat f^*\sZ-\sZ\in\CH_{\dim
  B'}(\sF')$, by assumption it satisfies $\Gamma|_{F'_{b'}}=0$ in $\CH_0(F'_{b'})$ for a general point $b'\in
  B'$. However, by an argument of Hilbert scheme (\cf \cite[Chapter
  22]{MR1997577}), the set of points $b'\in B'$ satisfying
  $\Gamma|_{F_b'}=0$ is a countable union of closed algebraic
  subsets. Thus together with the preimage of the locus where $b\in B$ does not satisfy the condition in the lemma (which is again a countable union of proper closed algebraic subsets by assumption), they cover $B'$. By Baire theorem, in this countable collection there exists one which is in fact the entire $B'$, \ie $\Gamma|_{F'_{b'}}=0$ holds for every $b'\in B'$.
  In particular, for $b'_0$, we have $\hat f^*Z-Z=\Gamma|_{F_{b_0}}=\Gamma|_{F'_{b'_0}}=0$.
\end{proof}
\begin{rmk}\label{rmk:shrinkB}
Thanks to this lemma, instead of defining $B$ as the parameter space
of smooth cubic fourfolds, we can and we will feel free to shrink
$B$ to any of its Zariski open subsets whenever we want to in the
rest of the paper.
\end{rmk}

To summarize this section, we reduce the main theorem \ref{thm:main}
into the following statement:

\begin{thm}\label{thm:main2}
Let $n=p^m$ be a power of a prime number, $f$ be an automorphism of
$\P^5$ given by (\ref{eqn:f}): $$f: [x_0: x_1:\cdots: x_5] \mapsto
[\zeta^{e_0}x_0:\zeta^{e_1}x_1:\cdots: \zeta^{e_5}x_5],$$ and $j\in
\Z/n\Z$ be a solution to (\ref{eqn:j}): $e_0+e_1+\cdots+e_5=2j \mod
n.$\\ If for a general point $b\in \bar B:=
\P\left(\bigoplus_{\underline \alpha\in \Lambda_j}\C\cdot \underline
x^{~\underline \alpha}\right)$,  $X_b$ is smooth, then $f$ acts as
identity on $\CH_1(X_b)_\Q$, where $$\Lambda_j:=\left\{\underline
\alpha=(\alpha_0, \cdots, \alpha_5)\in
\N^5~|~ \substack{\alpha_0+\cdots+\alpha_5=3\\
e_0\alpha_0+\cdots +e_5\alpha_5=j \mod n }\right\}$$ as in
(\ref{eqn:Lambdaj}).
\end{thm}
\begin{proof}[Theorem \ref{thm:main2} $\Rightarrow$ Theorem \ref{thm:main}]
First of all, in order to prove the main theorem \ref{thm:main}, we
can assume that the order of $f$ is a power of a prime number:
suppose the prime factorization of the order of $f$ is
$$n=p_1^{a_1}p_2^{a_2}\cdots p_r^{a_r}.$$ Let
$g_i=f^{np_i^{-a_i}}$ for any $1\leq i\leq r$, then $g_i$ is
of order $p_i^{a_i}$. Since $\hat f$ acts symplectically on $F(X)$, so
do the $g_i$'s. Then by assumption, $\hat g_i$ acts as identity on
$\CH_0(F(X))$ for any $i$. Now by Chinese remainder theorem,
there exist $b_1, \cdots, b_r\in \N$ such that
$f=\prod_{i=1}^rg_i^{b_i}$. Therefore, $\hat f=\prod_{i=1}^r\hat
g_i^{b_i}$ acts as identity on $\CH_0(F(X))$ as well. Secondly, the
parameter space $\bar B$ comes from the constraints we deduced in
Lemma \ref{lemma:symp}. Thirdly, Lemma
\ref{lemma:ReductionToGeneral} allows us to reduce the statement to
the case of a (very) general member in the family. Fourthly, we can
switch from $\CH_0(F_b)$ to $\CH_1(X_b)_\Q$ by Corollary
\ref{cor:reduction}. Finally, the reformulation in terms of the polarization is explained in Proposition \ref{prop:natural}.
\end{proof}

%

\section{The Chow group of the total
space}\label{sect:totalspace}

As a key step toward the proof of Theorem \ref{thm:main2}, we establish
in this section the following result.

\begin{prop}\label{prop:total2} Consider the direct system consisting of the open subsets $B$ of $\bar B$, then we have
 $$\indlim_B\CH^4(\sX\times_B\sX)_{\Q, hom}=0.$$ More precisely, for an open subset $B$ of $\bar B$, and for any homologically trivial codimension 4 $\Q$-cycle $z$ of $\sX\times_B\sX$, there exists a dense open subset $B'\subset B$, such that the restriction of $z$ to the base changed family $\sX'\times_{B'}\sX'$ is rationally equivalent to 0.
\end{prop}

We achieve this in two steps: the first one is to show that
homological equivalence and rational equivalence coincide on a
resolution of singularities of the compactification
$\bar\sX\times_{\bar B}\bar \sX$ (see Proposition \ref{cor:sX2}
below); in the second step, to pass to the open variety
$\sX\times_B\sX$, we need to `extend' a homologically trivial cycle
of the open variety to a cycle homologically trivial of the
compactification or rather its resolution (see Proposition \ref{prop:extension} below). More
precisely, let $B$ be an open subset of $\bar B$:

\begin{prop}[Step 1]\label{cor:sX2}
There exists a resolution of singularities $\tau: W\to
\bar\sX\times_{\bar B}\bar\sX$, such that the rational equivalence
and homological equivalence coincide on $W$ when tensored with $\Q$
(see Definition \ref{def:P} below). In particular, $\CH^4(W)_{\Q,
hom}=0$.
\end{prop}

\begin{prop}[Step 2]\label{prop:extension}
Let $\tau: W\to\bar\sX\times_{\bar B}\bar\sX$ be a resolution of singularities.
For any homologically trivial cycle $z\in \CH^4(\sX\times_B\sX)_{\Q,
hom}$, there exist a dense open subset $B'\subset B$ and a
\emph{homologically trivial} cycle $\bar z\in
\CH^4(W)_{\Q,hom}$, such that
\begin{equation}\label{eqn:success}
 z|_{\sX'\times_{B'}\sX'}=\tau'_*\left(\bar z|_{W'}\right)\in \CH^4(\sX'\times_{B'}\sX')_\Q,
\end{equation}
where $\sX'=\sX\times_BB'$, $W'=W\times_{\bar B}B'$ are obtained by base change. We denote by $\tau'\colon  W'\to \sX'\times_{B'}\sX'$ the restriction of $\tau$ to $W'$.
\end{prop}

Proposition \ref{cor:sX2} and Proposition \ref{prop:extension} will
be proved in Subsection \S\ref{subsect:compactification} and
Subsection \S\ref{subsect:extension} respectively. Admitting them,
Proposition \ref{prop:total2} becomes obvious:
\begin{proof}[Prop. \ref{cor:sX2}+Prop. \ref{prop:extension} $\Rightarrow$ Prop. \ref{prop:total2}]
Let $W, \tau$ be as in Proposition \ref{cor:sX2}. For a given $z\in \CH^4(\sX\times_B\sX)_{\Q,
hom}$,  let $B', \sX', \bar z, \tau', W'$ be as in Proposition \ref{prop:extension}. 
Since $\bar z$ is homologically trivial, $\bar z$ is (rationally equivalent to) zero by Proposition
\ref{cor:sX2}, hence so is its restriction $\bar z|_{W'}$ to the open subset $W'$. Therefore by
(\ref{eqn:success}),
$$z|_{\sX'\times_{B'}\sX'}=\tau'_*\left(\bar z|_{W'}\right)=0.$$
\end{proof}

%

\subsection{The Chow group of the compactification}\label{subsect:compactification}
In this subsection, we prove Proposition \ref{cor:sX2}.

We first recall the following notion due to Voisin
\cite[\S2.1] {MR3099982}:
\begin{defi}\label{def:P}
 We say a smooth projective variety $X$ satisfies \emph{property $\sP$}, if
the cycle class map is an isomorphism $$[-]: \CH^*(X)_\Q\lra{\isom}
H^*(X, \Q).$$
\end{defi}
Here we provide some examples and summarize some operations that
preserve this property $\sP$. For details, \cf  \cite{MR3099982}.
\begin{lemma}\label{lemma:P}
\begin{itemize}
  \item[$(i)$] Homogeneous variety of the form $G/P$ satisfies
  property $\sP$, where $G$ is a linear algebraic group and $P$ is a
  parabolic subgroup. For example, projective spaces, Grassmannians,
  flag varieties, \etc.
  \item[$(ii)$] If $X$ and $Y$ satisfy property $\sP$, then so does $X\times
  Y$.
  \item[$(iii)$] If $X$ satisfies property $\sP$, and $E$ is a vector
  bundle on it, then the projective bundle $\P(E)$ satisfies
  property $\sP$.
  \item[$(iv)$] If $X$ satisfies property $\sP$, and $Z\subset X$ is a
   smooth subvariety satisfying property $\sP$, then so is the blow up
  variety $\Bl_ZX$.
  \item[$(v)$] Let $f: X\to X'$ be a surjective generic finite
  morphism. If $X$ satisfies property $\sP$, then so does $X'$.
\end{itemize}
\end{lemma}

Since some toric geometry will be needed in the sequel, let us also
recall some standard definitions and properties, see
\cite{MR1234037}, \cite{MR2810322} for details. Given a lattice $N$
and a fan $\Delta$ in $N_\R$, one can construct a toric variety of
dimension $\rank(N)$, which will be denoted by $X(\Delta)$. By
definition, $X(\Delta)$ is the union of affine toric varieties
$\Spec(\C[N\dual\cap \sigma\dual])$, where $N\dual$ is the dual
lattice, $\sigma\dual$ is the dual cone in $N_\R\dual$ and $\sigma$
runs over the cones in $\Delta$. A fan $\Delta$ is said
\emph{regular} if each cone in $\Delta$
  is generated by a part of a $\Z$-basis of $N$.
 Let $N'$ be another lattice and $\Delta'$ be a fan in $N'_\R$. Then a
homomorphism (as abelian groups) $f:N\to N'$ induces a rational map
of the toric varieties $\phi: X(\Delta)\dra X(\Delta')$. Such maps
are called \emph{equivariant} or \emph{monomial}.
\begin{prop}\label{prop:toric}
Using the above notation for toric geometry, then we have:
\begin{itemize}
  \item[$(i)$] $X(\Delta)$ is smooth if and only if $\Delta$ is regular.
  \item[$(ii)$] $\phi: X(\Delta)\dra X(\Delta')$ is a morphism if and
  only if for any cone $\sigma\in \Delta$, there exists a cone
  $\sigma'\in\Delta'$ such that $f$ sends $\sigma$ into $\sigma'$.
  \item[$(iii)$] Any fan admits a refinement consisting of regular cones.
  \item[$(iv)$] Any smooth projective toric variety satisfies property
  $\sP$.
  \item[$(v)$] $\phi$ admits an elimination of indeterminacies:
  \begin{displaymath}
    \xymatrix{X(\tilde\Delta)\ar[d]^{\tau} \ar[dr]^{\tilde\phi}&\\
    X(\Delta) \ar@{-->}[r]^{\phi} &X(\Delta')
    }
  \end{displaymath}
  such that $X(\tilde\Delta)$ is smooth projective satisfying
  property $\sP$.
\end{itemize}
\end{prop}
\begin{proof}
For $(i)$, \cite[Theorem 3.1.19]{MR2810322}; for $(ii)$,
\cite[Theorem 3.3.4]{MR2810322}; for $(iii)$, \cite[Theorem
11.1.9]{MR2810322}; for $(iv)$, \cite[Theorem 12.5.3]{MR2810322}.
Finally, $(v)$ is a consequence of the first four: by $(iii)$, we
can find a regular refinement of $\Delta\cup f^{-1}(\Delta')$,
denoted by $\tilde\Delta$, then $X(\tilde\Delta)$ is smooth by $(i)$
and satisfies property $\sP$ by $(iv)$. Moreover, $(ii)$ implies
that $\phi\circ \tau: X(\tilde\Delta)\to X(\Delta')$ is a morphism.
\end{proof}


Turning back to our question, we adopt the previous notation as in
Theorem \ref{thm:main2}.

We can view $\bar B=\P\left(\bigoplus_{\underline \alpha\in
\Lambda_j}\C\cdot \underline x^{~\underline \alpha}\right)$ as an
\emph{incomplete} linear system on $\P^5$ associated to the line
bundle $\sO_{\P^5}(3)$. We remark that by construction in
\S\ref{sect:setting}, each member of $\bar B$ (which is a possibly
singular cubic fourfold) is preserved under the action of $f$.
Consider the rational map associated to this linear system:
$$\phi:=\phi_{|\bar B|}: \P^5\dra \bar B\dual,$$ where $\bar B\dual$
is the dual projective space consisting of the hyperplanes of $\bar
B$. We remark that since $\bar B\dual$ has a basis given by
monomials, the above rational map $\phi$ is a \emph{monomial} map
between two toric varieties (\cf the definition before Proposition
\ref{prop:toric}).
\begin{lemma}\label{lemma:total1}
\begin{itemize}
  \item[$(i)$] There exists an elimination of indeterminacies
of  $\phi$:
  \begin{displaymath}
    \xymatrix{ \tilde{\P^5} \ar[d]_{\tau} \ar[dr]^{\tilde \phi} &
    \\
    \P^5 \ar@{-->}[r]^{\phi} & \bar B\dual
    }
  \end{displaymath}
 such that $\tilde{\P^5}$ is smooth projective satisfying property $\sP$.
  \item[$(ii)$] The strict transform of $\bar\sX\subset \P^5\times\bar
  B$, denoted by $\tilde{\bar\sX}$, is the incidence subvariety in $\tilde{\P^5}\times\bar
  B$:
  $$\tilde{\bar\sX}=\left\{(x,b)\in \tilde{\P^5}\times \bar B~|~ b\in \tilde\phi(x)\right\}.$$
\end{itemize}
\end{lemma}
\begin{proof}
  $(i)$ By Proposition \ref{prop:toric}$(v)$.\\
  $(ii)$ follows from the fact that for $x\in \P^5$ not in the base
  locus of $\bar B$, $b\in\phi(x)$ if and only if $(x,b)\in \bar\sX$.
\end{proof}

\begin{cor}\label{cor:sX}
$\tilde{\bar \sX}$ is smooth and satisfies property $\sP$.
\end{cor}
\begin{proof}
  Thanks to Lemma \ref{lemma:total1}$(iii)$, $\tilde{\bar\sX}\subset\tilde{\P^5}\times\bar
  B$ is the incidence subvariety with respect to $\tilde\phi:\tilde{\P^5}\to \bar
  B\dual.$ Therefore the first projection $\tilde{\bar\sX}\to
  \tilde{\P^5}$ is a projective bundle (whose fiber over $x\in\tilde{\P^5}$ is the hyperplane of $\bar
  B$ determined by $\tilde\phi(x)\in \bar B\dual$), hence smooth. By Lemma
  \ref{lemma:P}$(iii)$, $\tilde{\bar\sX}$ satisfies property $\sP$.
\end{proof}

%

We remark that the action of $f$ on $\P^5$ lifts to $\tilde{\P^5}$
because the base locus of $\bar B$ is clearly preserved by $f$.
Correspondingly, the linear system $\bar B$ pulls back to
$\tilde{\P^5}$ to a base-point-free linear system, still denoted by
$\bar B$, with each member preserved by $f$, and the morphism
$\tilde \phi$ constructed above is exactly the morphism associated
to this linear system.

To deal with the (possibly singular) variety $\bar\sX\times_{\bar
B}\bar\sX$, we follow the same recipe as before (see Lemma
\ref{lemma:total2}, Lemma \ref{lemma:total3} and Proposition
\ref{cor:sX2}). The morphism $\tilde\phi\times\tilde\phi:
\tilde{\P^5}\times\tilde{\P^5}\to \bar B\dual\times \bar B\dual$
induces a rational map
$$\varphi: \tilde{\P^5}\times\tilde{\P^5}\dra{} \Bl_{\Delta_{\bar
B\dual}}\left(\bar B\dual\times\bar
  B\dual\right).$$
  We remark that this rational map $\varphi$ is again monomial, simply because $\phi: \P^5\dra \bar B\dual$ is so.

\begin{lemma}\label{lemma:total2}

There exists an elimination of indeterminacies of $\varphi$:
    \begin{displaymath}
      \xymatrix{
      \tilde{\tilde{\P^5}\times\tilde{\P^5}} \ar[d]_{\tau}
      \ar[r]^-{\tilde{\varphi}} & \Bl_{\Delta}(\bar B\dual\times\bar B\dual) \ar[d]\\
\tilde{\P^5}\times\tilde{\P^5} \ar@{-->}[ur]^{\varphi}
\ar[r]^{\tilde\phi\times\tilde\phi} & \bar B\dual\times\bar
  B\dual
      }
    \end{displaymath}
    such that $\tilde{\tilde{\P^5}\times\tilde{\P^5}}$ is smooth projective satisfying property $\sP$.

\end{lemma}
\begin{proof}
It is a direct application of Proposition \ref{prop:toric}$(v)$.
\end{proof}

Consider the rational map $\bar B\dual\times \bar
B\dual\dra{}\Gr(\bar B,2)$ defined by `intersecting two
hyperplanes', where $\Gr(\bar B,2)$ is the Grassmannian of
codimension 2 sub-projective spaces of $\bar B$. Blowing up the
diagonal will resolve the indeterminacies:
\begin{displaymath}
  \xymatrix{
  \Bl_{\Delta_{\bar B\dual}}\left(\bar B\dual\times\bar
  B\dual\right) \ar[dr] \ar[d] &\\
  \bar B\dual\times \bar B\dual \ar@{-->}[r] & \Gr(\bar B, 2)
  }
\end{displaymath}
Composing it with $\tilde{\varphi}$ constructed in the previous
lemma, we obtain a morphism $$\psi:
\tilde{\tilde{\P^5}\times\tilde{\P^5}} \to \Gr(\bar B,2).$$ As in
Lemma \ref{lemma:total1}, we have
\begin{lemma}\label{lemma:total3}
Consider the following incidence subvariety of
$\tilde{\tilde{\P^5}\times\tilde{\P^5}}\times \bar
  B$ with respect to $\psi$:
  $$W:=\left\{(z,b)\in \tilde{\tilde{\P^5}\times\tilde{\P^5}}\times \bar B ~|~
  b\in\psi(z)\right\}.$$
\begin{itemize}
\item[$(i)$] The first projection $W\to
\tilde{\tilde{\P^5}\times\tilde{\P^5}}$ is a projective bundle,
whose fiber over $z\in \tilde{\tilde{\P^5}\times\tilde{\P^5}}$ is
the codimension 2 sub-projective space determined by $\psi(z)\in
\Gr(\bar B,2)$.
\item[$(ii)$] $W$ has a birational morphism onto $\bar\sX\times_{\bar
  B}\bar\sX$.
\end{itemize}
\end{lemma}
\begin{proof}
$(i)$ is obvious.\\
$(ii)$ We have a natural morphism $W\to \P^5\times\P^5\times\bar B$.
We claim that this morphism is birational onto its image $\bar\sX\times_{\bar B}\bar\sX$:
since for two general points $x_1, x_2$ in $\P^5$, $\psi(x_1,
x_2)=\phi(x_1)\cap\phi(x_2)$, thus $(x_1, x_2, b)\in W$ is by
definition equivalent to $b\in \phi(x_1)\cap\phi(x_2)$, which is
equivalent to $(x_1, x_2, b)\in \bar\sX\times_{\bar B}\bar\sX$. 
\end{proof}

Now we can accomplish our first step of this section:
\begin{proof}[Proof of Proposition \ref{cor:sX2}]
  Since $W$ is a projective bundle (Lemma \ref{lemma:total3}$(i)$)
  over the variety $\tilde{\tilde{\P^5}\times\tilde{\P^5}}$ satisfying property
  $\sP$ (Lemma \ref{lemma:total2}$(ii)$), $W$ satisfies also property
  $\sP$ (Lemma \ref{lemma:P}(iii)). Then we conclude by Lemma
  \ref{lemma:total3}$(ii)$.
\end{proof}

\subsection{Extension of homologically trivial algebraic cycles}\label{subsect:extension} In this subsection we prove Proposition \ref{prop:extension}.

To pass from the compactification $\bar\sX\times_{\bar B}\bar\sX$ to
the space $\sX\times_B\sX$ which concerns us, we would like to
mention Voisin's `conjecture N' (\cite[Conjecture 0.6] {MR3099982}):
\begin{conj}[Conjecture N]\label{conj:N}
Let $X$ be a smooth projective variety, and let $U\subset X$ be an
open subset. Assume an algebraic cycle $Z\in \CH^i(X)_\Q$ has
cohomology class $[Z]\in H^{2i}(X, \Q)$ which vanishes when
restricted to $H^{2i}(U, \Q)$. Then there exists another cycle
$Z'\in \CH^i(X)_\Q$, which is supported on $X\backslash U$ and such
that $[Z']=[Z]\in H^{2i}(X, \Q)$.
\end{conj}
This Conjecture N is equivalent to the surjectivity of $\CH^i(X)_{\Q,hom} \to \CH^i(U)_{\Q,hom}$, hence implies the following conjecture (\cf \cite[Lemma 4.20]{MR3186044}):
\begin{conj}\label{conj:open}
Let $X$ be a smooth projective variety, and $U$ be an open subset of
$X$. If $\CH^i(X)_{\Q,hom}=0$, then $\CH^i(U)_{\Q,hom}=0$.
\end{conj}

According to this conjecture, Proposition \ref{cor:sX2} would have
implied the desired result Proposition \ref{prop:total2}. To get
around this conjecture N, our starting point is the following
observation in \cite[Lemma 1.1] {MR3099982}.
\begin{lemma}\label{lemma:codim2}
Conjecture N is true for $i\leq2$. In particular, for $i\leq 2$ and
for any $Z^o\in \CH^i(U)_{\Q, hom}$, there exists $W\in
\CH^i(X)_{\Q, hom}$ such that $W|_U=Z^o$.
\end{lemma}

Now the crucial observation is that \emph{the Chow motive of a cubic
fourfold does not exceed the size of Chow motives of surfaces}, so
that we can reduce the problem to a known case of Conjecture N,
namely Lemma \ref{lemma:codim2}. To illustrate, we first investigate
the situation of one cubic fourfold (absolute case), then make the
construction into the family version.

\vspace{0.4cm} \noindent\textbf{Absolute case:}

Let $X$ be a (smooth) cubic fourfold. Recall the following diagram
as in the proof of the unirationality of cubic fourfold:
\begin{displaymath}
 \xymatrix{
\P\left(T_X|_{L}\right) \ar@{-->}[r]^-{q} \ar[d] & X\\
L & }
\end{displaymath}
Here we fix a line $L$ contained in $X$, and the vertical arrow is
the natural $\P^3$-fibration, and the rational map $q$ is defined in
the following classical way: for any $(x,v)\in \P(T_X|_L)$ where $v$
is a non-zero tangent vector of $X$ at $x\in L$, then as long as the
line $\P(\C\cdot v)$ generated by the tangent vector $v$ is not
contained in $X$, the intersection of this line $\P(\C\cdot v)$ with
$X$ should be three (not necessarily distinct) points with two of
them $x$. Let $y$ be the remaining intersection point. We define
$q:(x,v)\mapsto y$. By construction, the indeterminacy locus of $q$
is $\left\{(x, v)\in \P(T_X|_L)~|~ \P(\C\cdot v)\subset X\right\}$.
Note that $q$ is dominant of degree 2.

By Hironaka's theorem, we have an elimination of indeterminacies:
\begin{displaymath}
 \xymatrix{
\tilde{\P\left(T_X|_L\right)} \ar[d]_{\tau} \ar@{->>}[dr]^{\tilde q}&\\
\P\left(T_X|_L\right) \ar[d] \ar@{-->}[r]^{q} & X\\
L& }
\end{displaymath}
where $\tau$ is the composition of a series of successive blow ups
along smooth centers of dimension $\leq 2$, and $\tilde q$ is
surjective thus generically finite (of degree 2).

We follow the notation of \cite{MR2115000} to denote the category of
Chow motives with rational coefficients by $\CHM_\Q$, and to write
$\h$ for the Chow motive of a smooth projective variety, which is a
\emph{contravariant} functor
$$\h: \SmProj^{op} \to \CHM_\Q.$$

Denote $M:=\P\left(T_X|_L\right)$ and $\tilde
M:=\tilde{\P\left(T_X|_L\right)}$. Let  $S_i$ be the blow up centers
of $\tau$ and $c_i=\codim S_i\in \{2,3,4\}$.  By the blow up formula
and the projective bundle formula for Chow motives (\cf
\cite[4.3.2]{MR2115000}),
$$\h\left(\tilde M\right)=  \h\left(\P\left(T_X|_L\right)\right)\oplus\bigoplus_i\bigoplus_{l=1}^{c_i-1}\h(S_i)(-l)= \left(\bigoplus_{l=0}^3\h(L)(-l)\right) \oplus\left(\bigoplus_i\bigoplus_{l=1}^{c_i-1}\h(S_i)(-l)\right),$$
and since $L\isom \P^1$,
\begin{equation}\label{eqn:motivefib}
\h\left(\tilde M\right)=\left(\1\oplus\1(-1)^{\oplus
2}\oplus\1(-2)^{\oplus 2}\oplus\1(-3)^{\oplus
2}\oplus\1(-4)\right)\oplus\left(\bigoplus_i\bigoplus_{l=1}^{c_i-1}\h(S_i)(-l)\right),
\end{equation}
where $\1:=\h(\pt)$ is the trivial motive. On the other hand, since
$\tilde q_*{\tilde q}^*=\deg(\tilde q)=2\cdot\id$, $\h(X)$ is a
direct factor of $\h\left(\tilde{\P\left(T_X|_L\right)}\right)$,
which has been decomposed in (\ref{eqn:motivefib}). This gives a
precise explanation of what we mean by saying that \emph{$\h(X)$
does not exceed the size of motives of surfaces} above.

By the monoidal structure of $\CHM_\Q$ (\cf
\cite[4.1.4]{MR2115000}), the motive of $\tilde M\times\tilde M$ has
the following form:
\begin{equation}\label{eqn:motive2fib}
 \h\left(\tilde M\times\tilde M\right)= \bigoplus_{k\in J} \h(V_{k}\times
 V'_{k})(-l_k),
\end{equation}
where $J$ is the index set parameterizing all possible products, and
$V_{k}\times V'_{k}$ is of one of the following forms:
\begin{itemize}
  \item $\pt\times \pt$ and $l_k=0$ or $1$;
  \item $\pt\times S_{i}$ or $S_{i}\times \pt$ and $l_k=1$;
  \item $l_k\geq 2$.
\end{itemize}

For each summand $\h(V_{k}\times V_{k}')(-l_k)$ in
(\ref{eqn:motive2fib}), the inclusion of this direct factor
\begin{equation}\label{eqn:iotaabs}
 \iota_{k}\in \Hom_{\CHM_\Q}\left(\h(V_{k}\times V_{k}')(-l_k), \h\left(\tilde
M\times\tilde M\right)\right)
\end{equation}
determines a natural correspondence from $V_{k}\times V_{k}'$ to
$\tilde M\times \tilde M$.
Similarly, for each $k\in J$, the projection to the $k$-th direct
factor
\begin{equation}\label{eqn:pabs}
p_{k}\in \Hom_{\CHM_\Q}\left(\h\left(\tilde M\times\tilde
M\right),\h(V_{k}\times V_{k}')(-l_k) \right)
\end{equation}
determines also a natural correspondence from $\tilde M\times \tilde
M$ back to $V_{k}\times V_{k}'$.
By construction
$$p_{k}\circ \iota_{k}=\id \in \End_{\CHM_\Q}\left(\h(V_{k}\times V_{k}')(-l_k)\right),
~~~\text{for any }~ k\in J;$$
$$\sum_{k\in J}\iota_{k}\circ p_{k}=\id \in \End_{\CHM_\Q}\left(\h(\tilde M\times\tilde M)\right).$$
Equivalently, the last equation says
\begin{equation}\label{eqn:sumto1abs}
 \sum_{k\in J}\iota_{k}\circ p_{k}=\Delta_{\tilde M\times \tilde M}  ~~\text{in } \CH^*(\tilde M\times\tilde M\times\tilde M\times\tilde
 M)_\Q.
\end{equation}

\vspace{0.4cm} \noindent\textbf{Construction in family:}

We now turn to the family version of the above constructions. To
this end, we need to choose a specific line for each cubic fourfold
in the family, and also a specific point on the chosen line.
Therefore a base change (\ie $T\to B$ constructed below) will be
necessary to construct the family version of the previous $p_{k}$
and $\iota_{k}$'s (see Lemma \ref{lemma:iotap}).

Precisely, consider the universal family $\sX$ of cubic fourfolds
over $B$, and the universal family of Fano varieties of lines $\sF$
as well as the universal incidence varieties $\sP$:
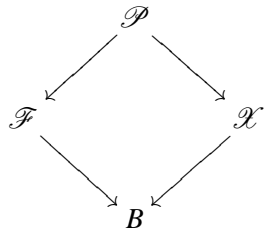
\begin{displaymath}
 \xymatrix{
&\sP\ar[dr] \ar[dl]&\\
\sF \ar[dr] & & \sX\ar[dl]\\
&B& }
\end{displaymath}
 By taking general hyperplane sections of $\sP$, we get $T$ a subvariety of it, such that the induced morphism $T\to B$ is generically finite. In fact, by shrinking\footnote{Recall that we are allowed to shrink $B$ whenever we want, see Remark \ref{rmk:shrinkB}.} $B$ (and also $T$ correspondingly), we can assume $T\to B$ is finite and \'etale, and hence $T$ is smooth.

By base change construction, we have over $T$ a universal family of
cubic fourfolds $\sY$, a universal family of lines $\sL$ contained
in $\sY$ and a section $\sigma: T\to \sL$ corresponding to the
universal family of the chosen points in $\sL$. We summarize the
situation in the following diagram:
\begin{equation}\label{diag:basechange}
 \xymatrix{
\sL\ar@{^(->}[r] \ar[dr] & \sY\ar[d]^{\pi'} \ar[r]^{r} \cart & \sX\ar[d]^{\pi}\\
& T\ar[r] \ar@/^/[ul]^{\sigma} & B }
\end{equation}
where for any $t\in T$ with image $b$ in $B$, the fiber $Y_t=X_b$,
$L_t$ is a line contained in it and $\sigma(t)\in L_t$. As $T\to B$
is finite and \'etale, so is $r: \sY\to \sX$.

Now we define $$\sM:= \P\left(T_{\sY/T}|_{\sL}\right),$$ and a
dominant rational map of degree 2 $$\xymatrix{q:\sM\ar@{-->}[r]
&\sY}.$$ Over $t\in T$, the fiber of $\sM$ is
$M_t=\P\left(T_{Y_t}|_{L_t}\right)$, and the restriction of $q$ to
this fiber is exactly the rational map $
\P\left(T_{Y_t}|_{L_t}\right) \dra Y_t$ constructed before in the
absolute case.

By Hironaka's theorem, we have an elimination of indeterminacies of
$q$:
\begin{equation}\label{diag:tildeM}
 \xymatrix{
\tilde\sM \ar[d]_{\tau} \ar[dr]^{\tilde q}& \\
\sM\ar[d] \ar@{-->}[r]^{q} & \sY \ar[ddl]\\
\sL \ar[d] & \\
T }
\end{equation}
such that, up to shrinking $B$ (and also $T$ correspondingly),
$\tau$ consists of blow ups along smooth centers which are
\emph{smooth} over $T$ (by generic smoothness theorem) of relative
dimension (over $T$) at most 2. Suppose the blow up centers are
$\sS_i$, whose codimension is denoted by $c_i\in \{2,3,4\}$. A fiber
of $\tilde \sM$, $\sM$, $\sS_i$ is exactly $\tilde M$, $M$, $S_i$
respectively constructed in the absolute case.
In the same fashion, we denote by $\sV_k$ and $\sV_k'$ the family
version of the varieties $V_k$ and $V_k'$ in the absolute case,
which is nothing else but of the form $\sS_i\times_T \sS_j$ or
$T\times_T\sS_i$ \etc. Let $d_k$ (\emph{resp}.\ $d_k'$) be the
dimension of $V_{k}$ (\emph{resp}.\ $V_{k}'$), \ie when $V_{k}=\pt$,
$\sV_k=\sigma(T)$ and $d_k=0$; when $V_{k}=S_{i}$, $\sV_k=\sS_i$ and
$d_k=4-c_i$, similarly for $\sV_k'$.

We can now globalize the correspondences (\ref{eqn:iotaabs}) and
(\ref{eqn:pabs}) into their following family versions. Here we use
the same notation:
\begin{lemma}\label{lemma:iotap}
\begin{itemize}
 \item[$(i)$]For any $k\in J$, there exist natural correspondences (over $T$)
$$\iota_k\in \CH^{d_k+d_k'+l_k}\left(\sV_k\times_T \sV_k'\times_T
\tilde \sM\times_T \tilde \sM\right)_\Q,$$
$$p_k\in \CH^{8-l_k}\left(\tilde \sM\times_T \tilde \sM\times_T
\sV_{k}\times_T \sV_{k}'\right)_\Q,$$
such that the following two identities hold on each fiber: for any
$t\in T$, we have
$$\left(p_k\circ \iota_{k}\right)_t=\Delta_{V_{k,t}\times V_{k,t}'} ~~~~\text{ for any }~k\in J;$$
\begin{equation}\label{eqn:sumto1}
 \sum_{k\in J}\left(\iota_{k}\circ p_{k}\right)_t=\Delta_{\tilde M_t\times \tilde M_t}.
\end{equation}
\item[$(ii)$] Up to shrinking $B$ and $T$ correspondingly, we have in $\CH^*(\tilde \sM\times_T\tilde \sM\times_T\tilde \sM\times_T\tilde \sM)_\Q$,
 \begin{equation}\label{prop:sumto1}
 \sum_{k\in J}\iota_{k}\circ p_{k}=\Delta_{\tilde \sM\times_T \tilde
 \sM}.
\end{equation}
\end{itemize}
\end{lemma}
\begin{proof}
$(i).$ For the existence, it suffices to remark that the
correspondences $\iota_{k}$ and $p_{k}$ can in fact be universally
defined over $T$, because when we make the canonical isomorphisms
(\ref{eqn:motivefib}) or (\ref{eqn:motive2fib}) precise by using the
projective bundle formula and blow up formula, they are just
compositions of inclusions, pull-backs, intersections with the
relative $\sO(1)$ of
projective bundles, each of which can be defined in family over $T$. Note that in this step of the construction, we used the section $\sigma$ to make the isomorphism $\h(L_t)\isom \1\oplus \1(-1)$ well-defined in family over $T$, because this isomorphism amounts to choose a point on the line.\\
Finally, equality (\ref{eqn:sumto1}) is exactly equality (\ref{eqn:sumto1abs}) in the absolute case.\\
$(ii).$ Equation (\ref{prop:sumto1}) is a direct consequence of
(\ref{eqn:sumto1}), thanks to the Bloch-Srinivas type argument of
spreading rational equivalences (\cf \cite{MR714776},
\cite[Corollary 10.20]{MR1997577}).
\end{proof}


Keeping the notation in Diagram (\ref{diag:basechange}) and Diagram
(\ref{diag:tildeM}), we consider the generic finite morphism
$$\tilde q\times\tilde q: \tilde \sM\times_T\tilde \sM\to
\sY\times_T\sY,$$ and the finite \'etale morphism $$r\times
r:\sY\times_T\sY\to \sX\times_B\sX.$$ For each $k\in J$, composing
the graphs of these two morphisms with $\iota_k$, we get a
correspondence over $B$ from $\sV_k\times_T \sV_k'$ to
$\sX\times_B\sX$:
\begin{equation}\label{eqn:gamma1}
\Gamma_k:=\Gamma_{r\times r}\circ\Gamma_{\tilde q\times\tilde
q}\circ\iota_k\in \CH^{d_k+d_k'+l_k}\left(\sV_k\times_T
\sV_k'\times_B \sX\times_B \sX\right)_\Q;
\end{equation}
similarly, composing $p_k$ with the transposes of their graphs, we
obtain a correspondence over $B$ in the other direction:
\begin{equation}\label{eqn:gamma2}
\Gamma'_k:=p_k\circ \left({}^t\Gamma_{\tilde q\times\tilde
q}\right)\circ\left({}^t\Gamma_{r\times r}\right)\in
\CH^{8-l_k}\left(\sX\times_B \sX\times_B \sV_{k}\times_T
\sV_{k}'\right)_\Q.
\end{equation}

\begin{lemma}\label{prop:gammacomposition}
The sum of compositions of the above two correspondences satisfies:
$$\sum_{k\in J}\Gamma_k\circ \Gamma'_k=4\deg(T/B)\cdot \id,$$ as correspondences from $\sX\times_B \sX$ to itself.
\end{lemma}
\begin{proof}
 It is an immediate consequence of equation (\ref{prop:sumto1}) and the projection formula (note that $\deg(r\times r)=\deg(T/B)$ and $\deg(\tilde q\times\tilde q)=4$).
\end{proof}

For any $k\in J$, fix a smooth projective compactification
$\bar{\sV_{k}\times_T \sV_{k}'}$ of $\sV_{k}\times_T \sV_{k}'$ such
that the composition $\bar{\sV_{k}\times_T \sV_{k}'}\dra{} T\to B\to
\bar B$ is a morphism. Recall that $\bar\sX\times_{\bar B}\bar\sX$
is a (in general singular) compactification of $\sX\times_B \sX$ and $\tau\colon W\to \bar\sX\times_{\bar B}\bar\sX$ is a resolution of singularities. Put $W^o:=W\times_{\bar B}B$, and $\tau^o\colon W^o\to \sX\times_B \sX$ the restriction of $\tau$. (See the end of this section for a diagram.)
Consider the composition of the correspondence $\Gamma_k \in\CH^{d_k+d_k'+l_k}\left(\sV_k\times_T
\sV_k'\times_B \sX\times_B \sX\right)_\Q$ constructed above, with the pull-back by $\tau^o$ viewed as a correspondence ${}^t\Gamma_{\tau^o}\in \CH^8\left(\sX\times_B\sX\times_B W^o\right)$, we have a correspondence 
$${}^t\Gamma_{\tau^o}\circ\Gamma_k \in\CH^{d_k+d_k'+l_k}\left(\sV_k\times_T
\sV_k'\times_B W^o\right)_\Q.$$
Taking its closure, we obtain a correspondence between their smooth compactifications:
$$\bar{{}^t\Gamma_{\tau^o}\circ\Gamma_k}\in \CH^{d_k+d_k'+l_k}\left(\bar{\sV_k\times_T
\sV_k'}\times_{\bar B} W\right)_\Q.$$

For some technical reasons in the proof below, we have to separate
the index set $J$ into two parts and deal with them differently.
Recall that below equation (\ref{eqn:motive2fib}), we observed
that there are three types of elements $(V_{k}\times V_{k}', l_k)$
in $J$. Define the subset consisting of elements of the third type:
$$J':=\left\{k\in J~|~ (V_{k}\times V_{k}', l_k)~\text{ satisfies }~ l_k\geq 2\right\}.$$
And define $J''$ to be the complement of $J'$: elements of the first
two types. Note that for any $k\in J''$, the corresponding
$(V_{k}\times V_{k}', l_k)$ satisfies always
\begin{equation}\label{eqn:J''}
 \dim(V_{k}\times V_{k}')< 4-l_k ~~~\text{ for any }~ k\in J''.
\end{equation}

We can now accomplish the main goal of this subsection:

\begin{proof}[Proof of Proposition \ref{prop:extension}]
Let $z\in \CH^4(\sX\times_B\sX)_{\Q,hom}$. To simplify the notation,
we will omit the lower star for the correspondences $\Gamma_k$,
$\Gamma_k'$ and $\bar{{}^t\Gamma_{\tau^o}\circ\Gamma_k}$ since we never use
their transposes.


An obvious remark: when $k\in J''$, for any $b\in B$, the fiber
${\left(\Gamma_k'\right)_t}(z_b)\in \CH^{4-l_k}(V_{k,t}\times
V'_{k,t})_\Q=0$ by dimension reason (\cf (\ref{eqn:J''})), thus
\begin{equation}\label{eqn:J''trivial}
 \left(\sum_{k\in J''}{\Gamma_k}\circ{\Gamma_k'}(z)\right)_b=0.
\end{equation}

As a result, for any $b\in B$, in $\CH^4(X_b\times X_b)_\Q$,
\begin{equation}\label{eqn:proof}
\left(\sum_{k\in
J'}{\Gamma_k}\left({\Gamma_k'}(z)\right)\right)_b=\left(\sum_{k\in
J}{\Gamma_k}\left({\Gamma_k'}(z)\right)\right)_b=4\deg(T/B)\cdot
z_b,
\end{equation}
where the first equality comes from (\ref{eqn:J''trivial}) and the
second equality is by Lemma \ref{prop:gammacomposition}.

On the other hand, ${\Gamma_k'}(z)\in \CH^{4-l_k}\left(\sV_k\times_T
\sV_k'\right)_{\Q,hom}$ is homologically trivial. We claim that for
any $k\in J'$, the cycle ${\Gamma_k'}(z)$ `extends' to a
\emph{homologically trivial} algebraic cycle in the
compactification, \ie there exists $\xi_k\in
\CH^{4-l_k}\left(\bar{\sV_k\times_T \sV_k'}\right)_{\Q,hom}$ such
that $\xi_k|_{\sV_k\times_B \sV_k'}={\Gamma_k'}(z)$. Indeed, since
$4-l_k\leq 2$ for $k\in J'$ and $\sV_k\times_T \sV_k'$ is smooth by
construction, we can apply Lemma \ref{lemma:codim2} to find $\xi_k$.

Now for any $k\in J'$, let us consider the cycle $\bar{{}^t\Gamma_{\tau^o}\circ\Gamma_k}
\left(\xi_k\right)\in \CH^4(W)_{\Q,hom}$. Its
fiber over a point $b\in B$ is:
$$\bar{{}^t\Gamma_{\tau^o}\circ\Gamma_k}\left(\xi_k\right)_b=\left({{}^t\Gamma_{\tau^o}\circ\Gamma_k}(\xi_k|_{\sV_k\times_B \sV_k'})\right)_b=\left({}^t\Gamma_{\tau^o}\circ{\Gamma_k}\circ{\Gamma_k'}(z)\right)_b.$$
Therefore by (\ref{eqn:proof}), the restrictions of the following two cycles\footnote{Here ${\tau^o}^*$ is well-defined cause $W^o$ and $\sX\times_B\sX$ are both smooth.} to a fiber $W_b$
$$\beta:=\sum_{k\in J'}{\bar{{}^t\Gamma_{\tau^o}\circ\Gamma_k}}(\xi_k)~\text{ and }~
4\deg(T/B)\cdot {\tau^o}^*(z)$$ are the same in $\CH^4(W_b)_\Q$ for
any $b\in B$, where $W_b:=W\times_B \{b\}$. Again by the argument of Bloch and Srinivas (\cf
\cite{MR714776}, \cite[\S10.2]{MR1997577}, \cite[\S2]{MR3186044}),
there exists a dense open subset $B'\subset B$, such that
\begin{equation} \label{eqn:beta}
 \beta|_{W'}=\left(4\deg(T/B)\cdot {\tau^o}^*(z)\right)|_{W'}=\tau'^*\left(4\deg(T/B)\cdot z|_{\sX'\times_{B'}\sX'}\right)\in \CH^4(W')_\Q,
\end{equation}
where$\sX'=\sX\times_BB'$, $W'=W\times_{\bar B}B'$ are obtained by base change:
\begin{displaymath}
\xymatrix{
W'\ar@{^(->}[r] \ar[d]_{\tau'}\cart & W^o \ar@{^(->}[r]\cart \ar[d]^{\tau^o} & W\ar[d]^{\tau}\\
\sX'\times_{B'}\sX' \cart\ar@{^(->}[r]\ar[d]&\sX\times_B\sX \ar@{^(->}[r]\ar[d] \cart& \bar\sX\times_{\bar B}\bar\sX\ar[d]\\
B'\ar@{^(->}[r] & B\ar@{^(->}[r] & \bar B
}
\end{displaymath}

 Define $\bar
z:= \frac{1}{4\deg(T/B)}\beta\in \CH^4(W)_{\Q}$. By (\ref{eqn:beta}) and the projection formula for $\tau'$, we have the required property (\ref{eqn:success}) in Proposition \ref{prop:extension}: $$\tau'_*\left(\bar z|_{W'}\right)=\tau'_*\tau'^*\left(z|_{\sX'\times_{B'}\sX'}\right)=z|_{\sX'\times_{B'}\sX'}\in \CH^4(\sX'\times_{B'}\sX')_\Q,$$
 To conclude Proposition \ref{prop:extension}, it
suffices to remark that $\bar z$ is by construction homologically
trivial: since the $\xi_k$'s are homologically trivial, so is
$\beta=\sum_{k\in J'}{\bar{{}^t\Gamma_{\tau^o}\circ\Gamma_k}}(\xi_k)$ and hence $\bar z$.
\end{proof}

\section{Proof of Theorem \ref{thm:main2}}\label{sect:proof}
The content of this section is the proof of Theorem \ref{thm:main2}.

We keep the notation $n$, $f$, $j$, $\Lambda_j$, $\bar B$ as in the
statement of Theorem \ref{thm:main2}. Let $B$ be an open subset of
$\bar B$ parameterizing smooth cubic fourfolds. In the following
diagram,
\begin{equation}
 \xymatrix{X_b\ar@(dl, ul)^{f_b} \ar@{^(->}[r] \ar[d]\cart & \sX\ar@(dr, ur)_{f}\ar[d]^{\pi}\\
b \ar@{^(->}[r] & B}
\end{equation}
$\pi$ is the universal family equipped with a fiberwise action $f$;
for each $b\in B$, we write $f_b$ the restriction of $f$ on the
cubic fourfold $X_b$ if we need to distinguish it from $f$. By
construction, for any $b\in B$, $f_b$ is an automorphism of $X_b$ of
order $n$ which acts as identity on $H^{3,1}(X_b)$.

To begin with, we study the Hodge structures of the fibers:
\begin{lemma}\label{lemma:fibreHS} For any $b\in B$,
\begin{itemize}
 \item[$(i)$] $f^*$ is an order $n$ automorphism of the Hodge structure $H^4(X_b, \Q)$ and $f_*=\left(f^*\right)^{-1}=\left(f^*\right)^{n-1}$.
 \item[$(ii)$] There is a direct sum decomposition into sub-Hodge structures
\begin{equation}\label{eqn:decomp}
 H^4(X_b, \Q)=H^4(X_b, \Q)^{inv} \oplus^{\perp} H^4(X_b, \Q)^\#,
\end{equation}
where the first summand is the $f^*$-invariant part, and the second
summand is its orthogonal complement with respect to the
intersection product $<-,->$.
 \item[$(iii)$] $H^{3,1}(X_b)\subset H^4(X_b, \C)^{inv}$.
 \item[$(iv)$] $H^4(X_b, \Q)^\#$ is generated by the classes of some codimension 2 algebraic cycles.
\end{itemize}
\end{lemma}
\begin{proof}
 $(i)$ is obvious since $f^*$ must preserve the Hodge structure. The last equality comes from $f_*f^*=f^*f_*=\id$ and $\left(f^*\right)^n=\id$.\\
$(ii)$. Since $f^*$ is of finite order, it is semi-simple: $H^4(X_b, \Q)$ decomposes as direct sum of eigenspaces, where $H^4(X_b, \Q)^{inv} $ corresponds to eigenvalue 1 and $H^4(X_b, \Q)^\#$ is the sum of other eigenspaces. Moreover, $f^*$ preserves the intersection pairing $<-,->$, thus the invariant eigenspace is orthogonal to any other eigenspace.\\
$(iii)$. This is our assumption that $f^*$ acts as identity on $H^{3,1}(X)$.\\
$(iv)$. By $(iii)$, $H^4(X_b, \Q)^\#\subset
H^{3,1}_{\R}(X_b)^{\perp}=H^{2,2}_{\R}(X_b)$, \ie $H^4(X_b, \Q)^\#$ is
generated by rational Hodge classes of degree 4. However, the Hodge
conjecture is known to be true for cubic fourfolds
(\cite{MR0453741}). We deduce that $H^4(X_b, \Q)^\#$ is generated by
the classes of some codimension 2 subvarieties in $X_b$.
\end{proof}

Define the algebraic cycle
\begin{equation}\label{eqn:projectorInv}
\pi^{inv}:=\frac{1}{n}\left(\Delta_{\sX}+\Gamma_{f}+\cdots+\Gamma_{f^{n-1}}\right)\in
\CH^4(\sX\times_B \sX)_\Q.
\end{equation}
Here $\pi^{inv}$ can be viewed as a family of self-correspondences
of $X_b$ parameterized by $B$, more precisely:
\begin{equation}\label{eqn:projectorInvb}
\pi^{inv}|_{X_b\times
X_b}=:\pi^{inv}_b=\frac{1}{n}\left(\Delta_{X_b}+\Gamma_{f_b}+\cdots+\Gamma_{f_b^{n-1}}\right)\in
\CH^4(X_b\times X_b)_\Q.
\end{equation}
It is now clear that $\pi^{inv}_b$ acts on $H^4(X_b, \Q)$ by:
$$[\pi^{inv}_b]^*=\frac{1}{n}\left(Id+{f_b^*}+\cdots+\left(f_b^*\right)^{n-1}\right),$$ which is exactly the orthogonal projector onto the invariant part in the direct sum decomposition (\ref{eqn:decomp}). On the other hand, $\pi^{inv}_b$ acts  as identity on $H^0(X_b,\Q)$, $H^2(X_b,\Q)$, $H^6(X_b,\Q)$, $H^8(X_b,\Q)$.

Define another cycle
\begin{equation}\label{eqn:Gamma}
\Gamma:=\Delta_\sX-\pi^{inv}\in \CH^4(\sX\times_B\sX)_\Q.
\end{equation}
Then $\Gamma_b:=\Gamma|_{X_b\times
X_b}=\Delta_{X_b}-\frac{1}{n}\left(\Delta_{X_b}+\Gamma_{f_b}+\cdots+\Gamma_{f_b^{n-1}}\right)\in
\CH^4(X_b\times X_b)_\Q$ acts on $H^4(X_b, \Q)$ as the orthogonal
projector onto $H^4(X_b, \Q)^\#$ and acts as zero on $H^0(X_b,\Q)$,
$H^2(X_b,\Q)$, $H^6(X_b,\Q)$, $H^8(X_b,\Q)$. Now we have some
control over the cohomology class of `fibers' of $\Gamma$:
\begin{prop}\label{prop:fiberwise} For any $b\in B$,
\begin{itemize}
 \item[$(i)$]Let $\Gamma_b:=\Gamma|_{X_b\times X_b}\in \CH^4(X_b\times X_b)_\Q$. Then its cohomology class $[\Gamma_b]\in H^8(X_b\times X_b, \Q)$ is contained in $H^4(X_b, \Q)^\#\otimes H^4(X_b, \Q)^\#$.
 \item[$(ii)$] The cohomology class $[\Gamma_b]$ is supported on $Y_b\times Y_b$, where $Y_b$ is a closed algebraic subset of $X_b$ of codimension at least 2.
 \item[$(iii)$] Moreover, there exists an algebraic cycle $\Gamma_b'\in \CH^4(X_b\times X_b)_\Q$, which is supported on $Y_b\times Y_b$ and such that $[\Gamma_b]=[\Gamma'_b]$ in $H^8(X_b\times
 X_b,\Q)$
\end{itemize}
\end{prop}
\begin{proof}
$(i)$. By K\"unneth formula, we make the
identification 
\begin{multline*}
H^8(X_b\times X_b)=\left(H^0\otimes H^8\right)\oplus \left(H^2\otimes
H^6\right)\oplus \left(H^6\otimes H^2\right)\oplus \left(H^8\otimes H^0\right)\\
\oplus \left(H^{4,inv}\otimes H^{4,inv}\right)\oplus \left(H^{4,inv}\otimes H^{4,\#}\right)\oplus \left(H^{4,\#}\otimes
H^{4,inv}\right)\oplus \left(H^{4,\#}\otimes H^{4,\#}\right).\\
\end{multline*}
By construction and Poincar\'e duality, the cohomology class $[\Gamma_b]$ can only have the last component non-zero.\\
$(ii)$ is a consequence of $(iii)$.\\
$(iii)$. By Lemma \ref{lemma:fibreHS}$(iv)$, $H^4(X_b, \Q)^\#$ is
generated by the classes of some codimension 2 subvarieties in
$X_b$. We thus assume
$H^4(X_b,\Q)^\#=\Q[W_1]\oplus\cdots\oplus\Q[W_r]$, where $W_i$'s are
subvarieties of codimension 2 in $X_b$.  We can now take $Y_b:=\bigcup_{i=1}^rW_i$ to be the
codimension 2 closed algebraic subset, and take
$\Gamma_b'$ to be of the form $\Sigma_{i, j=1}^rb_{ij}W_i\times W_j\in \CH^4(X_b\times
X_b)_\Q$.
\end{proof}
Roughly speaking, the previous proposition says that when restricted
to each fiber, the cycle $\Gamma$ becomes homologically equivalent
to a cycle supported on a codimension 2 algebraic subset of the
fiber. Now here comes the crucial proposition, which allows us to
get some global information about $\Gamma$ from its fiberwise
property. The proposition appeared in Voisin's paper
 \cite{MR3099982}:
\begin{prop}\label{prop:crucial}
  In the above situation as in Proposition
  \ref{prop:fiberwise}, there exist a closed algebraic
  subset $\sY$ in $\sX$ of codimension 2, and an algebraic cycle $\Gamma'\in
  \CH^4(\sX\times_B\sX)_\Q$ which is supported on $\sY\times_B\sY$,
  such that for any $b\in B$, $[\Gamma'|_{X_b\times
X_b}]=[\Gamma|_{X_b\times
  X_b}]$ in $H^8(X_b\times X_b,\Q)$.
\end{prop}
\begin{proof}
See  \cite{MR3099982} Proposition 2.7.
\end{proof}

Let $\sY\subset \sX$ be the codimension 2 closed algebraic subset
introduced above, and $\Gamma'\in \CH^4(\sX\times_B \sX)_\Q$ be the
cycle supported on $\sY\times _B\sY$, as constructed in Proposition
\ref{prop:crucial}. Define
\begin{equation}\label{def:Z}
  \sZ:=\Gamma-\Gamma'\in \CH^4(\sX\times_B \sX)_\Q.
\end{equation}
Then by construction, for any $b\in B$, the `fiber'
$Z_b:=\sZ|_{X_b\times X_b}\in \CH^4(X_b\times X_b)_\Q$ has trivial
cohomology class:
\begin{equation}\label{eqn:Zb}
  [Z_b]=0 \in H^8(X_b\times X_b, \Q),~~~~~~\text{for any}~~b\in B.
\end{equation}

The next step is to prove the following decomposition of the
projector (Proposition \ref{prop:decomp}) from the fiberwise
cohomological triviality of $\sZ$ in (\ref{eqn:Zb}).

\begin{prop}\label{prop:decomp}
There exist a closed algebraic subset $\sY$ in $\sX$ of codimension
2, an algebraic cycle $\Gamma'\in \CH^4(\sX\times_B\sX)_\Q$
supported on $\sY\times_B\sY$, $\sZ'\in \CH^4(\sX\times \P^5)_\Q$
and $\sZ''\in \CH^4(\P^5\times \sX)_\Q$ such that we have an
equality in $\CH^4(\sX\times_B\sX)_\Q$:
\begin{equation}\label{finalequation1}
\Delta_{\sX}-
\frac{1}{n}\left(\Delta_{\sX}+\Gamma_{f}+\cdots+\Gamma_{f^{n-1}}\right)=
\Gamma'+ \sZ'|_{\sX\times_B\sX}+\sZ''|_{\sX\times_B\sX}.
\end{equation}
\end{prop}

To pass from the fiberwise equality (\ref{eqn:Zb}) to the global
equality (\ref{finalequation1}) above, we have to firstly deduce
from (\ref{eqn:Zb}) some global equality up to homological
equivalence, then use the result of \S\ref{sect:totalspace} to get
an equality up to rational equivalence. The argument of Leray
spectral sequence due to Voisin \cite[Lemma 2.11] {MR3099982} (in
our equivariant case) can accomplish the first step.

By Deligne's theorem \cite{MR0244265}, the Leray spectral sequence
associated to the smooth projective morphism $\pi\times\pi:
\sX\times_B\sX\to B$  degenerates at $E_2$:
$$E_{\infty}^{p,q}=E_2^{p,q}=H^p(B, R^q(\pi\times \pi)_*\Q) \Rightarrow H^{p+q}\left(\sX\times_B\sX, \Q\right).$$ In other words,
$$\Gr_L^{p}H^{p+q}(\sX\times_B\sX, \Q)=H^p(B, R^q(\pi\times \pi)_*\Q),$$ where $L^\bullet$ is the resulting \emph{Leray filtration} on $H^*(\sX\times_B\sX, \Q)$. The property (\ref{eqn:Zb}) is thus equivalent to
\begin{lemma}\label{lemma:inL1}
The cohomology class $[\sZ]\in L^1H^8(\sX\times_B\sX,\Q)$.
\end{lemma}
\begin{proof}
The image of $[\sZ]\in H^8(\sX\times_B\sX,\Q)$ in the first graded
piece $\Gr_L^{0}H^8(\sX\times_B\sX, \Q)=H^0(B, R^8(\pi\times
\pi)_*\Q)$ is a section of the local system $R^8(\pi\times
\pi)_*\Q$, whose fiber over $b\in B$ is $H^8(X_b\times X_b, \Q)$.
The value of this section on $b$ is given exactly by $[Z_b] \in
H^8(X_b\times X_b, \Q)$, which vanishes by (\ref{eqn:Zb}). Therefore
$[\sZ]\in H^8(\sX\times_B\sX,\Q)$ becomes zero in the quotient
$\Gr_L^0$, hence is contained in $L^1$.
\end{proof}

Consider the Leray spectral sequences associated to the following
three smooth projective morphisms to the base $B$:
\begin{displaymath}
  \xymatrix{\sX\times \P^5 \ar[dr]_{\pi\circ\pr_1} &\sX\times_B\sX
  \ar@{_(->}[l]_{\id\times i} \ar@{^(->}[r]^{i\times\id} \ar[d]^{\pi\times\pi}& \P^5\times
  \sX\ar[dl]^{\pi\circ\pr_2} \\ & B &
  }
\end{displaymath}
and the restriction maps for cohomology induced by the two
inclusions. We have the following lemma, where all the cohomology
groups are of rational coefficients.
\begin{lemma}\label{lemma:decomp}
Let $L^\bullet$ be the Leray filtrations corresponding to the above
Leray spectral sequences.
\begin{itemize}
  \item[$(i)$] The K\"unneth isomorphisms induce canonical isomorphisms $$L^1H^8(\sX\times \P^5)=\bigoplus_{i+j=8}L^1H^i(\sX)\otimes
  H^j(\P^5),$$ $$L^1H^8(\P^5\times
  \sX)=\bigoplus_{i+j=8}H^i(\P^5)\otimes L^1H^j(\sX).$$
  \item[$(ii)$] The natural restriction map $$L^1H^8(\sX\times \P^5)\oplus L^1H^8(\P^5\times \sX)\surj L^1H^8(\sX\times_B\sX)$$ is surjective.
\end{itemize}
\end{lemma}
\begin{proof}
By snake lemma (or five lemma) and induction, it suffices to prove
the corresponding results for the graded pieces.\\
$(i)$ We only prove the first isomorphism, the second one is
similar. For any $p\geq 1$, the isomorphism $$\Gr_L^pH^8(\sX\times
\P^5)=\bigoplus_{i+j=8}\Gr_L^pH^i(\sX)\otimes
  H^j(\P^5)$$ by Deligne' theorem is equivalent to
$$H^p(B, R^{8-p}(\pi\circ \pr_1)_*\Q)=H^p\Big(B,
\bigoplus_{i+j=8}\left(R^{i-p}\pi_*\Q\right)\otimes_\Q
H^j(\P^5)\Big).$$ However, $R^{8-p}(\pi\circ \pr_1)_*\Q$ is a local system with fiber $H^{8-p}(X_b\times \P^5,\Q)$, which is by
K\"unneth formula isomorphic to $\oplus_{i+j=8}H^{i-p}(X_b)\otimes
H^j(\P^5)$, which is exactly the fiber of the local system
$\oplus_{i+j=8}\left(R^{i-p}\pi_*\Q\right)\otimes_\Q H^j(\P^5)$.
Thus $(i)$ is a consequence of the relative K\"unneth formula.\\
$(ii)$ Using $(i)$, for any $p\geq 1$, the surjectivity of
$\Gr_L^pH^8(\sX\times \P^5)\oplus \Gr_L^pH^8(\P^5\times \sX)\surj
\Gr_L^pH^8(\sX\times_B\sX)$ is by Deligne's theorem equivalent to
the surjectivity of
$$H^p\Big(B, \bigoplus_{i+j=8}\left(R^{i-p}\pi_*\Q\right)\otimes_\Q
H^j(\P^5)\Big)\oplus H^p\Big(B, \bigoplus_{i+j=8}H^i(\P^5)\otimes_\Q
\left(R^{j-p}\pi_*\Q\right) \Big)\surj H^p(B, R^{8-p}(\pi\times
\pi)_*\Q).$$  By relative K\"unneth isomorphism, $R^{8-p}(\pi\times
\pi)_*\Q=\bigoplus_{k+l=8-p}R^k\pi_*\Q\otimes R^l\pi_*\Q$. Since
$8-p\leq 7$, either $k<4$ or $l<4$. Recall that $H^{odd}(X_b)=0$ and
the restriction map $H^{2i}(\P^5)\to H^{2i}(X_b)$ is an isomorphism
for $i=0,1,3,4$, thus $R^{8-p}(\pi\times \pi)_*\Q$ is a direct
summand of the local system
$\bigoplus_{k+l=8-p}R^k\pi_*\Q\otimes_\Q H^l(\P^5)\oplus
\bigoplus_{k+l=8-p}H^k(\P^5)\otimes_\Q R^l\pi_*\Q$. Therefore the
above displayed morphism is induced by the projection of a local system to its direct summand, which is of course surjective on
cohomology.
\end{proof}
Combining Lemma \ref{lemma:inL1} and Lemma \ref{lemma:decomp}, we
can decompose the cohomology class $[\sZ]$ as follows:
\begin{equation}\label{eqn:decompZ}
[\sZ]=\sum_{i=0}^4\pr_1^*A_i\cdot
\pr_2^*[h]^{4-i}+\sum_{j=0}^4\pr_1^*[h]^{4-j}\cdot \pr_2^*B_j \in
H^8(\sX\times_B\sX, \Q),
\end{equation}
where $\pr_i:\sX\times_B\sX\to \sX$ is the $i$-th projection,
$A_i\in H^{2i}(\sX, \Q)$, $B_j\in H^{2j}(\sX, \Q)$ and $h\in
\CH^1(\sX)$ is the pull back by the natural morphism $\sX\to \P^5$
of the hyperplane divisor $c_1(\sO_{\P^5}(1))$.

We remark that $A_i$ and $B_j$ must be \emph{algebraic}, that is,
they are the cohomology classes of algebraic cycles of $\sX$. The
reason is very simple: $[\sZ]$ being algebraic, so is
$$\pr_{1,*}\left([\sZ]\cdot\pr_2^*[h]^i\right)=3A_i+(\text{a rational number})~[h]^i,$$
thus $A_i$ is algebraic. The algebraicity of $B_j$ is similar. We
denote still by $A_i\in \CH^i(\sX)_\Q$ and $B_j\in \CH^j(\sX)_\Q$
for the algebraic cycles with the respective cohomology classes.
Therefore (\ref{eqn:decompZ}) becomes
$$[\sZ]=\sum_{i=0}^4[\pr_1^*A_i\cdot
\pr_2^*h^{4-i}]+\sum_{j=0}^4[\pr_1^*h^{4-j}\cdot \pr_2^*B_j] \in
H^8(\sX\times_B\sX, \Q).$$ In other words, there exist algebraic
cycles $$\sZ':=\sum_{i=0}^4\pr_1^*A_i\cdot \pr_2^*h^{4-i}\in
\im\left(\CH^4(\sX\times \P^5)_\Q\to
\CH^4(\sX\times_B\sX)_\Q\right),~~~~~\text{and}$$
$$\sZ'':=\sum_{j=0}^4\pr_1^*h^{4-j}\cdot \pr_2^*B_j\in
\im\left(\CH^4(\P^5\times \sX)_\Q\to
\CH^4(\sX\times_B\sX)_\Q\right),$$ such that
\begin{equation}\label{eqn:decompZ2}
[\sZ]=[\sZ'+\sZ'']\in H^8(\sX\times_B\sX, \Q).
\end{equation}

This is an equality up to homological equivalence. Now enters the
result of \S\ref{sect:totalspace}: thanks to Proposition
\ref{prop:total2}, up to shrinking $B$ to a dense open subset (still
denoted by $B$), the cohomological decomposition
(\ref{eqn:decompZ2}) in fact implies the following equality up to
\emph{rational equivalence} in Chow groups:
$$\sZ=\sZ'|_{\sX\times_B\sX}+ \sZ''|_{\sX\times_B\sX} \in \CH^4(\sX\times_B\sX)_\Q.$$

Combining this with (\ref{eqn:projectorInv}), (\ref{eqn:Gamma}),
Proposition \ref{prop:crucial} and (\ref{def:Z}), we obtain a
decomposition of the projector (\ref{finalequation1}) announced in
Proposition \ref{prop:decomp}.

Now we can deduce Theorem \ref{thm:main2} from this decomposition as
follows. For any $b\in B$ (thus general in $\bar B$), taking the
fiber of (\ref{finalequation1}) over $b$, we get an equality in
$\CH^4(X_b\times X_b)_\Q$.
\begin{equation}\label{finaleqn2}
\Delta_{X_b}=\frac{1}{n}\left(\Delta_{X_b}+\Gamma_{f}+\cdots+\Gamma_{f^{n-1}}\right)+
\Gamma'_b+ Z'_b|_{X_b\times X_b}+Z''_b|_{X_b\times X_b},
\end{equation}
where we still write $f$ for $f_b$ the restriction of the action on
fiber $X_b$, $\Gamma'_b$ is supported on $Y_b\times Y_b$ with $Y_b$
a closed algebraic subset of codimension 2 in $X_b$, and $Z'_b$
(\emph{resp.}\ $Z''_b$) is a cycle of $X_b\times\P^5$ (\emph{resp.}\
$\P^5\times X_b$) with rational coefficients.

For any homologically trivial 1-cycle $\gamma\in
\CH_1(X_b)_{\Q,hom}$, let both sides of (\ref{finaleqn2}) act on it
by correspondences. We have in $\CH_1(X_b)_\Q$:
\begin{itemize}
  \item $\Delta_{X_b}^*(\gamma)=\gamma$;
  \item
  $\frac{1}{n}\left(\Delta_{X_b}+\Gamma_{f}+\cdots+\Gamma_{f^{n-1}}\right)^*(\gamma) =
  \frac{1}{n}\left(\gamma+f^*\gamma+\cdots+\left(f^*\right)^{n-1}\gamma\right)$;
  \item $\Gamma_b'^*(\gamma)=0$ because the support of
  $\Gamma_b'$ has the projection to the first coordinate codimension 2;
  \item $\left(Z'_b|_{X_b\times X_b}\right)^*\left(\gamma\right)=\left(Z''_b|_{X_b\times
  X_b}\right)^*(\gamma)=0$, since they both factorizes through $\CH^*(\P^5)_{\Q,
  hom}=0$.
\end{itemize}
As a result, we have in $\CH_1(X_b)_\Q$,
$$\gamma=\frac{1}{n}\left(\gamma+f^*\gamma+\cdots+\left(f^*\right)^{n-1}\gamma\right),$$
where the right hand side is obviously invariant by $f^*$, hence so
is the left hand side. In other words, $f^*$ acts on
$\CH_1(X_b)_{\Q, hom}$ as identity. Finally, since $H^6(X_b,\Q)$ is
1-dimensional with $f^*$ acting trivially, we have for any
$\gamma\in \CH_1(X_b)_{\Q}$, $$\pi^{inv,*}(\gamma)-\gamma\in
\CH_1(X_b)_{\Q,hom},$$ where
$\pi^{inv,*}=\frac{1}{n}\left(\id+f^*+\cdots+\left(f^*\right)^{n-1}\right)$.
Therefore by what we just obtained,
$$f^*\left(\pi^{inv,*}(\gamma)-\gamma\right)=\pi^{inv,*}(\gamma)-\gamma.$$
As $\pi^{inv,*}(\gamma)$ is obviously $f^*$-invariant, we have
$f^*(\gamma)=\gamma$ in $\CH_1(X_b)_\Q$. Theorem \ref{thm:main2}, as
well as the main Theorem \ref{thm:main}, is proved.

%


\section{A remark}\label{sect:natural}
In the main Theorem \ref{thm:main}, we assumed that the automorphism
of the Fano variety of lines is induced from an automorphism of the
cubic fourfold itself. In this section we want to reformulate
this hypothesis.

\begin{prop}\label{prop:natural}
Let $X\subset \P^5$ be a (smooth) cubic fourfold, and $\left(F(X),
\sL\right)$ be its Fano variety of lines equipped with the Pl\"ucker
polarization induced from the ambiant Grassmannian $\Gr(\P^1,
\P^5)$. An automorphism $\psi$ of $F(X)$ comes from an automorphism
of $X$ if and only if it is polarized (\ie $\psi^*\sL\isom\sL$).
\end{prop}
\begin{proof}
 The following proof is taken from \cite[Proposition 4]{RmkTorelli}.
Consider the projective embedding of $F(X)$ determined by the
Pl\"ucker polarization $\sL$:
$$F(X)=:F\subset \Gr(\P^1, \P^5)=:G\subset \P(\wedge^2H^0(\P^5, \sO(1))\dual)=:\P^{14}.$$
If $\psi$ is induced from an automorphism $f$ of $X$, which must be
an automorphism of $\P^5$, then it is clear that $\psi$ is the
restriction of a linear automorphism of $\P^{14}$, thus the
Pl\"ucker polarization is preserved.

Conversely, if the automorphism $\psi$ preserves the polarization,
it is then the restriction of a projective automorphism of
$\P^{14}$, which we denote still by $\psi$. It is proved in
\cite[1.16(iii)]{MR0569043} that $G$ is the intersection of all the
quadrics containing $F$. It follows that $\psi$ is an automorphism
of $G$, because $\psi$ sends any quadric containing $F$ to a quadric
containing $F$. However any automorphism of $G$ is induced by a
projective automorphism $f$ of $\P^5$ (\cf \cite{MR0028057}). As a
result, $f$ sends a line contained in $X$ to a line contained in
$X$, thus $f$ preserves $X$ and $\psi$ is induced from $f$.
\end{proof}

Define $\Aut^{pol}(F(X))$ to be the group of polarized automorphisms
of $F(X)$ (equipped with Pl\"ucker polarization $\sL$). Then the
previous proposition says that we have an isomorphism
\begin{eqnarray*}
 \Aut(X)&\lra{\isom}& \Aut^{pol}(F(X))\\
f&\mapsto& \hat f
\end{eqnarray*}
This gives us the last statement in
main Theorem \ref{thm:main}: \emph{any polarized symplectic
automorphism of $F(X)$ acts as identity on $\CH_0(F(X))$.}


%

\section{A consequence: action on
$\CH_2(F)_{\Q,hom}$}\label{sect:CH2} As an application of Theorem
\ref{thm:main}, we study in this section the induced action on the
Chow group of 2-cycles. The conclusion is Corollary \ref{cor:CH2} in
the introduction.

\begin{proof}[Proof of Corollary \ref{cor:CH2}]
 As showed in the previous section, the polarized automorphism $\hat f$ on
$F(X)$ comes from an automorphism $f$ of finite order $n$ of the
smooth cubic fourfold $X$. We consider again the projector
$\Gamma:=\Delta_F-\pi^{inv} \in \CH_4(F\times F)_\Q$, where
$$\pi^{inv}=\frac{\Delta_F+\Gamma_{\hat f}+\cdots+\Gamma_{{\hat f}^{n-1}}}{n}\in \CH_4(F\times
F)_\Q.$$ We remark that $\Gamma={}^t\Gamma$ since $\hat f^{-1}=\hat
f^{n-1}$. Our main result Theorem \ref{thm:main} says in particular
that the action of $\Gamma$ on $\CH_0(F)_\Q$ is zero:
$$\Gamma_*=0: \CH_0(F)_\Q\to \CH_0(F)_\Q.$$ Equivalently speaking,
the restriction of $\Gamma$ to each fiber is zero:
$$\Gamma|_{\{t\}\times F}=0 \in \CH_0(F)_\Q, ~~~\forall t\in F.$$ By
the argument of Bloch-Srinivas (\cf \cite{MR714776},
\cite[\S10.2]{MR1997577}), there exist an effective reduced divisor
$D\subsetneq X$, a resolution of singularities $\tau: \tilde D\to D$
and an algebraic cycle $\Gamma'\in \CH_4(\tilde D\times F)_\Q$ such
that $\Gamma=(\tilde\tau\times \id_F)_*\Gamma'$, where $\tilde\tau$
is the composition of $\tau$ and the inclusion of $D$ into $X$.
Consequently, the action of $\Gamma={}^t\Gamma$ on $\CH_2(F)_\Q$
factorises as:
\begin{displaymath}
  \xymatrix{
  \CH_2(F)_\Q \ar[dr]_{\Gamma'^*} \ar[rr]^{\Gamma_*=\Gamma^*} &&
  \CH_2(F)_\Q \\
  &\CH^1(\tilde D)_\Q \ar[ur]_{\tilde \tau_*}&
  }.
\end{displaymath}
Since these correspondences preserve the homological equivalence as
well as the Abel-Jacobi equivalence (\cf \cite[Chapter
9]{MR1997577}), we have in fact the following factorization:
\begin{displaymath}
  \xymatrix{
  \CH_2(F)_{\Q,AJ} \ar[dr]_{\Gamma'^*} \ar[rr]^{\Gamma_*=\Gamma^*} &&
  \CH_2(F)_{\Q,AJ} \\
  &\CH^1(\tilde D)_{\Q,AJ} \ar[ur]_{\tilde \tau_*}&
  }.
\end{displaymath}
where $AJ$ means the Abel-Jacobi kernels. However, it is well-known
that for divisors the Abel-Jacobi map is an isomorphism. Hence
$$\CH^1(\tilde D)_{\Q,AJ}=0.$$ Now the factorization implies that $\Gamma$ acts as zero on
$\CH_2(F)_{\Q,AJ}$, thus $\hat f$ acts as identity on
$\CH_2(F)_{\Q,AJ}$. To conclude, it suffices to remark that from the
vanishing $H^3(F)=0$ (because $F$ is deformation equivalent to the second punctual Hilbert scheme of K3 surfaces \cite{MR818549}, which has vanishing odd cohomology), the Abel-Jacobi kernel
$$\CH_2(F)_{\Q,AJ}:=\ker\left(\CH_2(F)_{\Q,hom}\to
J^3(F):=\frac{H^{0,3}(F)\oplus H^{1,2}(F)}{H^3(F,\Z)}=0\right)$$ is
equal to $\CH_2(F)_{\Q,hom}$.
\end{proof}

\begin{rmk}
We want to remark that Corollary \ref{cor:CH2} is also predicted by
Bloch-Beilinson conjecture (a more general version than Conjecture
\ref{conj:BB}, \cf \cite[Chapter 11]{MR1997577}, \cite[Chapitre
11]{MR2115000}).
\end{rmk}

\bibliographystyle{amsplain}
\bibliography{biblio_fulie}

\end{document}